\documentclass[12pt]{amsart}

\usepackage{amsmath,graphicx}
\usepackage{adjustbox}
\usepackage{amssymb}
\usepackage{tikz, tikz-cd}
\usepackage{fullpage}
\usepackage[all]{xy}
\usepackage[margin=1in]{geometry}
\usepackage{amsthm}
\usepackage{amsfonts}
\usepackage{bbm}
\usepackage{comment}
\usepackage{amsmath}
\usepackage{amsthm,diagbox}
\usepackage{bbm}
\usepackage{eufrak}
\usepackage{array} 
\usepackage{color}
\usepackage[utf8]{inputenc}
\usepackage[english]{babel}
\usepackage{hyperref}
\newcolumntype{L}{>{$}l<{$}}
\usepackage{bm}
\usepackage{subcaption} % successor package to subfig

\usepackage[pagewise]{lineno}
\usepackage{soul}

\DeclareMathSymbol{\shortminus}{\mathbin}{AMSa}{"39}

\newtheorem{theorem}{Theorem}[section]
\newtheorem{lemma}[theorem]{Lemma}

\newtheorem{corollary}[theorem]{Corollary}

\newtheorem{proposition}[theorem]{Proposition}
\theoremstyle{definition}

\newtheorem{remark} [theorem] {Remark}
\newtheorem{question} [theorem] {Question}

\theoremstyle{definition}

\newcommand{\C}{{\mathbb{C}}}
\newcommand{\F}{{\mathbb{F}}}

\newcommand{\Q}{{\mathbb{Q}}}

\newcommand{\Z}{{\mathbb{Z}}}

\newcommand{\Tor}{\text{Tor}}

\newcommand{\inj}{\hookrightarrow}

\newcommand{\lk}{\ell\text{k}}

\newcommand{\spin}{\text{spin}}

\newcommand{\slflk}{\mathit{sl}}
\newcommand{\slrep}{\mathfrak{sl}_2}

\DeclareMathOperator{\HFL}{HFL}
\DeclareMathOperator{\HFK}{HFK}

\DeclareMathOperator{\CKh}{CKh}
\DeclareMathOperator{\Kh}{Kh}
\DeclareMathOperator{\Khr}{Khr}

\DeclareMathOperator{\AKh}{AKh}
\DeclareMathOperator{\CAKh}{CAKh}
\DeclareMathOperator{\SFH}{SFH}

\DeclareMathOperator{\CF}{CF}

\DeclareMathOperator{\AHI}{AHI}

\DeclareMathOperator{\rank}{rank}

\date{\today}
\title{Closures of $3$-braids and Detection}
\author{Fraser Binns}
\date{\today}
\address[]{Department of Mathematics, Princeton University}
\email{fb1673@princeton.edu}
\thanks{FB was supported by the Simons Grant {\em New structures in low-dimensional topology}.}


\begin{document}

\begin{abstract}
  We give some new link detection results for link Floer homology, Khovanov homology and annular Khovanov homology. The links we detect arise via different closure operations on $3$-braids. Examples of our results include that link Floer homology detects the Mazur link, that annular Khovanov homology detects the Mazur pattern, and that Khovanov homology detects L6a2 and L9n15. The Mazur pattern detection result depends on a new bound on the rank of the annular Khovanov homology of certain links.
\end{abstract}

\maketitle

Braids are of wide mathematical interest; see the survey article~\cite{birman2005braids}. In this paper we will consider the four different types of links obtained from braids as shown in Figure~\ref{fig:closures}.\begin{center}
 \begin{figure}[h]

 \begin{subfigure}{0.4\textwidth}\centering
        \includegraphics[width=0.4\textwidth]{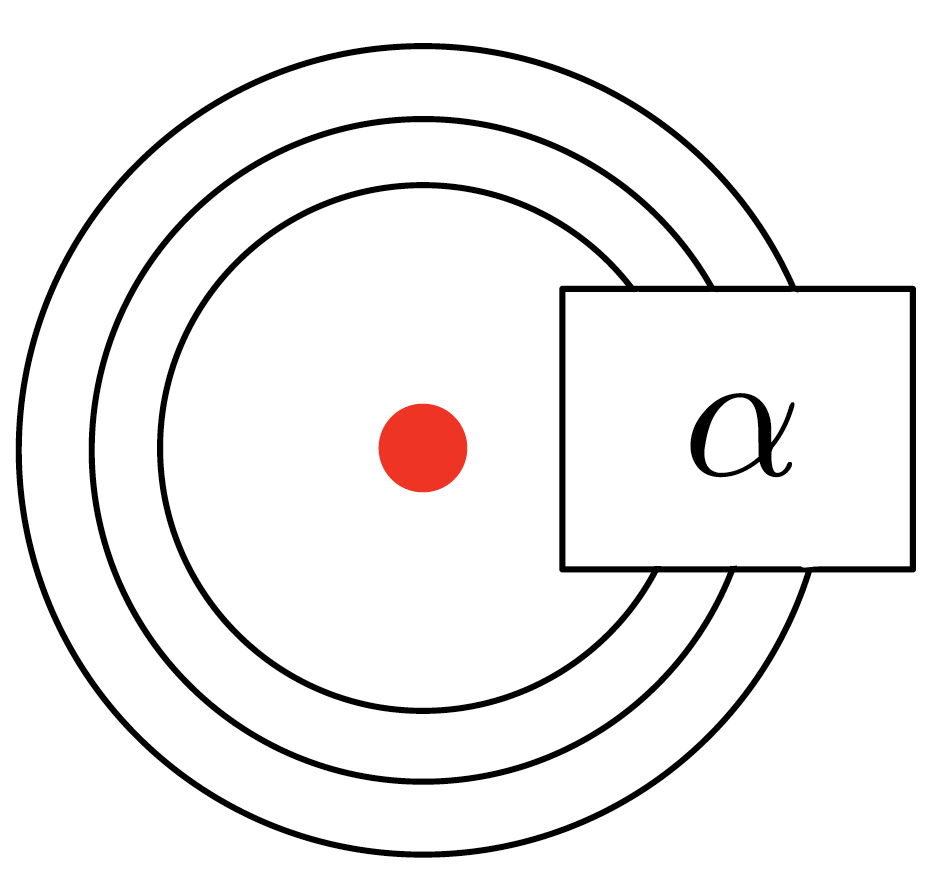}\caption{A braid-closure, $b(\alpha)$.}\label{fig:braidclosure}
    \end{subfigure}    \begin{subfigure}{0.4\textwidth}\centering
        \includegraphics[width=0.4\textwidth]{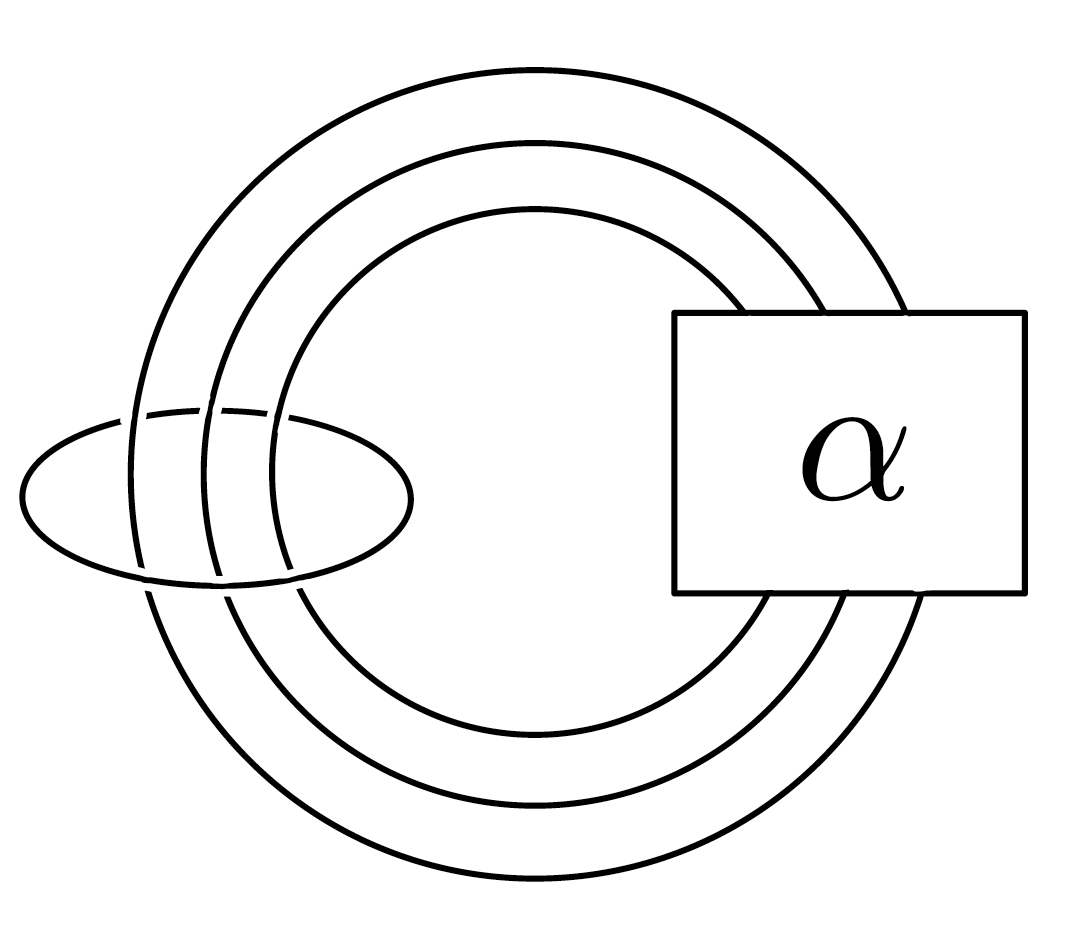}\caption{An augmented braid-closure, $\mathring{b}(\alpha)$.}\label{fig:augmentedbraidclosure}
    \end{subfigure}
    
    \begin{subfigure}{0.4\textwidth}\centering
        \includegraphics[width=0.4\textwidth]{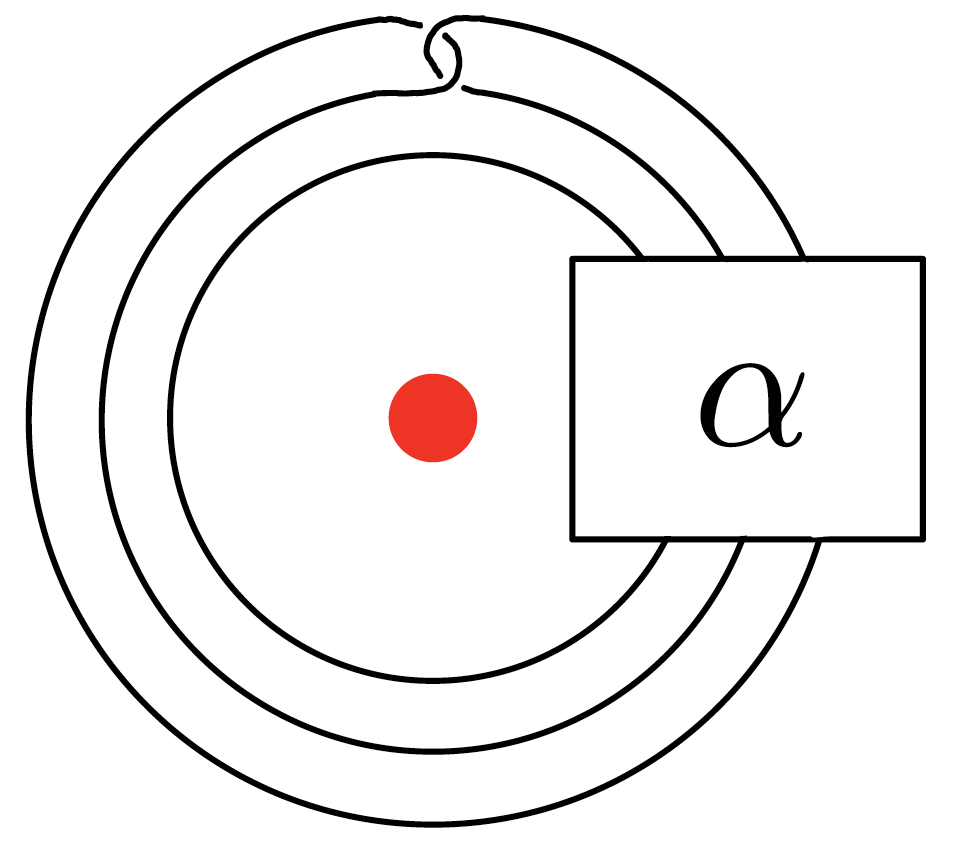}\caption{A clasp-closure, $c(\alpha)$.}\label{fig:claspclosure}
    \end{subfigure} \begin{subfigure}{0.4\textwidth}\centering
        \includegraphics[width=0.4\textwidth]{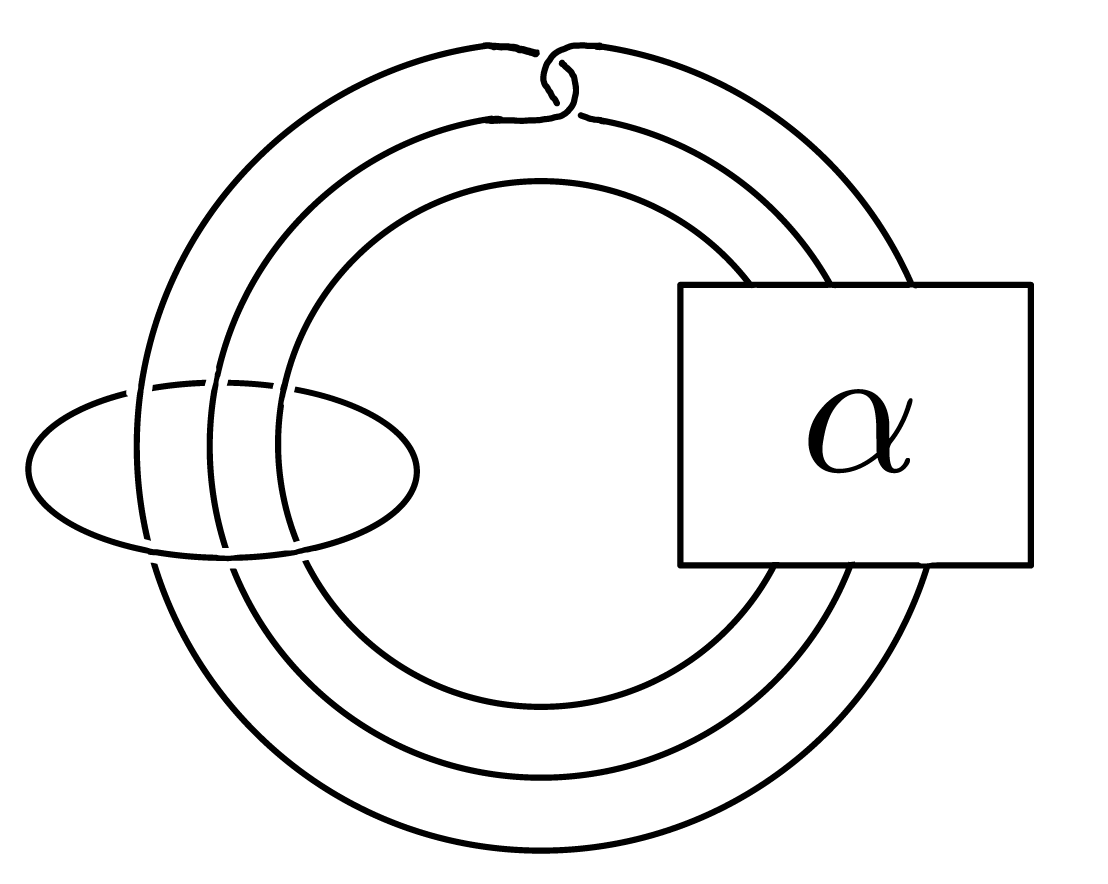}\caption{An augmented clasp-closure, $\mathring{c}(\alpha)$.}\label{fig:augmentedclaspclosure}
    \end{subfigure}

    \caption{The four types of links we study in this paper. The red dots indicates the axes.}\label{fig:closures}
   \end{figure}\end{center}Let $\alpha$ be a braid. The first two types of link we obtain from $\alpha$ have been widely studied. We have the \emph{braid-closure} of $\alpha$, $b(\alpha)$, which for the purposes of this paper is the links in the thickened annulus obtained by attaching $n$ parallel strands as in Figure~\ref{fig:braidclosure}. Secondly, we have the \emph{augmented braid-closure} of $\alpha$, $\mathring{b}(\alpha)$, which is the link obtained by adding the annular axis to $b(\alpha)$, as shown in Figure~\ref{fig:augmentedbraidclosure}.

For the remaining two types of link we move beyond the usual setting of braid-closures. The \emph{clasp-closure} of $\alpha$, $c(\alpha)$, can be thought of as the annular link formed by $b(\alpha)$ and replacing two parallel strands in a ball with a \emph{clasp}, as shown in Figure~\ref{fig:claspclosure}. Note that the clasp is between the two rightmost strands of $\alpha$, though this is simply a matter of convention and plays no significant role. The \emph{augmented clasp-closure} of $\alpha$, $\mathring{c}(\alpha)$, is defined analogously to the augmented braid-closure of $\alpha$, see Figure~\ref{fig:augmentedclaspclosure}.

As motivation for studying clasp-closures, recall that a result of Martin~\cite[Proposition 1]{martin2022khovanov}, states that the links with the simplest link Floer homology --- in an appropriate sense --- are augmented braid-closures. A result of the author and Dey showed that augmented clasp-closures are examples of links with second simplest link Floer homology, in the same sense~\cite[Theorem 5.1]{binns2024floer}. Thus one might reasonably expect that understanding the behaviour of categorified link invariants of  braid- and clasp-closures of braids might be easier than understanding other types of closure.

 (Augmented)  braid-closures of one and $2$-braids are readily classified up to isotopy. Braid-closures of $3$-braids were classified completely by Murasugi~\cite{Murasugiclosed3}. In particular, he showed that there are three braid-closures of $3$-braids representing the unknot, namely $\sigma_1\sigma_2, \sigma_1^{-1}\sigma_2^{-1}$ and $\sigma_1\sigma^{-1}_2$. Here we use the standard Artin generators for the braid group. The augmented braid-closures of these three links are $T(2,6)$, $T(2,-6)$ and $L6a2$ respectively. More generally Birman-Menasco showed that for $|n|\neq 1$ there are two $3$-braid representatives of the torus links $T(2,n)$, namely $\sigma_1^{\pm1}\sigma_2^n$~\cite{BirmanMenasco3braids}. Note that $\mathring{b}(\sigma_1\sigma_2^3)$ is $L9n15$ while $\mathring{b}(\sigma_1^{-1}\sigma_2^3)$ is $L9n16$.
 
 One cannot take the clasp-closure of a $1$-braid. Clasp-closures of $2$-braids are the twisted Whitehead patterns. The case of (augmented) clasp-closures of $3$-braids is more complicated. Baldwin-Sivek classified $3$-braids with clasp-closures representing the unknot, up to isotopy of the clasp-closure, see the proof of~\cite[Theorem 6.1]{baldwin2022floer}. Up to mirroring and reversal these braids are as follows:

\begin{enumerate}
    \item $\sigma_1^{-1}$. The augmentation of this link is L7a6, i.e. the mirror of the Mazur link.
   
    \item $\sigma_1^3\sigma_2^{-1}\sigma_1^2\sigma_2$.
     \item $\sigma^n_1\sigma_2^{-1}\sigma_1\sigma_2$. For $n=1$ the augmentation of this link{ }is L7a5.

\end{enumerate}

Our goal in this paper is to exploit Baldwin-Sivek, Murasugi and Birman-Menasco's classification results to obtain detection results for various categorified link invariants. There are three different invariants we will study; link Floer homology and two versions of Khovanov homology.

\emph{Link Floer homology} is an invariant of oriented links defined by Ozsv\'ath-Szab\'o using symplectic topology~\cite{ozsvath2008holomorphic}. For two{-}component links it takes value in the category of triply graded vector spaces. Our first result is the following:

\newtheorem*{thm:HFLL6a2}{Theorem~\ref{thm:HFLL6a2}}
\begin{thm:HFLL6a2}
    Link Floer homology detects L6a2.
\end{thm:HFLL6a2}

\newtheorem*{thm:HFLl9n15}{Theorem~\ref{thm:HFLl9n15}}
\begin{thm:HFLl9n15}
    Link Floer homology with rational coefficients detects L9n15.
\end{thm:HFLl9n15}

L6a2 is the augmented braid-closure of $\sigma_1\sigma_2^{-1}$. The author and Martin showed that link Floer homology detects the augmented braid-closures of the other two braids that represent the unknot i.e. it was shown that link Floer homology detects $T(2,\pm 6)$ endowed with any orientation~\cite{binns2020knot}. The author and Dey showed that link Floer homology detects all of the augmented braid-closures of $2$-braids~\cite{binns2022cable}. The augmented braid-closure of the $1$-braid is also detected by link Floer homology since it is simply a Hopf link.

The proof strategies for Theorem~\ref{thm:HFLL6a2} and Theorem~\ref{thm:HFLl9n15} are that used by the author and Martin in~\cite{binns2020knot}. That is we use the fact that the link Floer homology of a link $L$ contains various pieces of topological information about $L$. In particular we appeal to Martin's result that link Floer homology detects braid axes~\cite[Proposition 1]{martin2022khovanov}.

We now address $3$-braids with unknotted clasp-closures. For the first type we have detection.

\newtheorem*{thm:l7a6plus}{Theorem~\ref{thm:l7a6plus}}
\begin{thm:l7a6plus}
    {Link Floer homology detects the Mazur link.} %and $\mathring{c}(\sigma_1^3\sigma_2^{-1}\sigma_1^2\sigma_2)$.
\end{thm:l7a6plus}

The author and Dey showed that link Floer homology detects the augmented clasp-closures of all but two $2$-braids and that the remaining two augmented clasp-closures are the unique links of their link Floer homology type~\cite[Theorem 6.1, 6.2]{binns2024floer}.

For the final type of $3$-braid with unknotted clasp-closure we do not get detection. Nevertheless we can give the following classification result:

\newtheorem*{thm:infinitefamily}{Theorem~\ref{thm:infinitefamily}}
\begin{thm:infinitefamily}
    Let $L$ be a link. $\widehat{\HFL}(L)\cong\widehat{\HFL}(\mathring{c}(\sigma_2^{-1}\sigma_1\sigma_2))$ if and only if $L$ is of the form $\mathring{c}(\sigma_1^n\sigma_2^{-1}\sigma_1\sigma_2)$ for some $n\in\Z$.
\end{thm:infinitefamily}

The proof strategies for these two theorems are similar to that used in the proof of Theorem~\ref{thm:HFLL6a2}. The chief difference is that we appeal to the classification of links with link Floer homology of next to minimal rank in certain gradings~\cite[Theorem 5.1]{binns2024floer}, as opposed to Martin's braid axis detection result which was a classification of links with link Floer homology of minimal rank in certain gradings~\cite[Proposition 1]{martin2022khovanov}.

We now turn to \emph{Khovanov homology}. This is a combinatorial link invariant due to Khovanov that takes value in the category of bi-graded vector spaces~\cite{khovanov2000categorification}. We have the following two results:

\newtheorem*{thm:khl6a2}{Theorem~\ref{thm:khl6a2}}
\begin{thm:khl6a2}
    Khovanov homology with integer coefficients detects L6a2.
\end{thm:khl6a2}

\newtheorem*{thm:KhL9n15}{Theorem~\ref{thm:KhL9n15}}
\begin{thm:KhL9n15}
    Khovanov homology with integer coefficients detects $L9n15$.
\end{thm:KhL9n15}

For context recall that Khovanov homology detects the Hopf link~\cite{baldwin2018khovanovhopf}. It also detects the augmented link associated to all 2-braid representatives of the unknot --- namely $T(2,\pm 4)$. This was originally proven by using instanton Floer homology~\cite{xie2020links}, see also~\cite{binns2020knot} for a proof that is more in line with that of Theorem~\ref{thm:khl6a2}. Martin showed that Khovanov homlogy detects $T(2,6)$, one of the braid-closures of a $3$-braid representing the unknot.

The main tool we use to prove this result is Dowlin's spectral sequence~\cite{dowlin2018spectral} from Khovanov homology to knot Floer homology --- a version of link Floer homology due independently to Ozsv\'ath-Szab\'o~\cite{Holomorphicdisksandknotinvariants} and J. Rasmussen~\cite{Rasmussen}. This allows us to reduce the question of detection for Khovanov homology to problems in link Floer homology.

Finally we study \emph{annular Khovanov homology}, a version of Khovanov homology for links in the thickened annulus due to Asaeda-Przytycki-Sikora~\cite{asaeda_categorification_2004}. We have the following family of results:

\newtheorem*{thm:AKH}{Theorem~\ref{thm:AKH}}
\begin{thm:AKH}
 Annular Khovanov homology with integer coefficients detects $b(\sigma_1\sigma_2^{n})$ for $-2\leq n\leq 5$.
\end{thm:AKH}

For context, recall that annular Khovanov homology detects the braid-closure of the identity braids~\cite{BaldwinGrigsby}, {and} all braid-closures of $2$-braids by a combination of work of Grigsby-Ni~\cite{grigsby_sutured_2014} and Grigsby-Licata-Wehrli~\cite{grigsby_sutured_2014}. The author and Martin also showed th{e} $n=1$ case of Theorem~\ref{thm:AKH} {in}~\cite{binns2020knot}. For the proof of our result we use Birman-Menasco's classification of $3$-braids with fixed closures~\cite{BirmanMenasco3braids}.

We can also prove the following:

\newtheorem*{thm:AKHmazur}{Theorem~\ref{thm:AKHmazur}}
\begin{thm:AKHmazur}
Annular Khovanov homology with integer coefficients detects the Mazur pattern.
\end{thm:AKHmazur}

Note that annular Khovanov homology detects the clasp-closures of all $2$-braids, amongst annular knots~\cite[Theorem 8.1]{binns2024floer}. For the proofs of the two preceeding theorems we use a version of the following rank bound:

\newtheorem*{thm:rankboundforarbbraid}{Theorem~\ref{thm:rankboundforarbbraid}}
\begin{thm:rankboundforarbbraid}\footnote{Part one of this theorem was included in an unpublished note of the author written while he was a graduate student.}
If $\beta$ is an $n$-braid with $n\geq 2$ then:\begin{enumerate}
    \item $\rank(\AKh(b(\beta);\C))\geq 2n$.
\item  $\rank(\AKh(c(\beta);\C))\geq 4n$.
\end{enumerate}
\end{thm:rankboundforarbbraid}

This result is inspired by the proof of a structurally similar rank bound in knot Floer homology due to Baldwin-Vela-Vick~\cite{baldwin_note_2018}. The proof relies on the left orderability of the braid group. See Lemma~\ref{lem:nontivialrep} for the more technical version of the result that we apply to prove Theorem~\ref{thm:AKHmazur} and Theorem~\ref{thm:AKH}. A number of other consequences of Theorem~\ref{thm:rankboundforarbbraid} are noted in Section~\ref{sec:applications}. The clasp-closure statement version of Theorem~\ref{thm:rankboundforarbbraid} is perhaps more interesting because there is currently no analogous result in the link Floer homology context.

\begin{remark}
    Link Floer homology, Khovanov homology and {a}nnular Khovanov homology are invariant under overall orientation reversal. All of the detection and classification results in this paper are thus up to overall orientation reversal, if any relevant link and its reverse are distinct.
\end{remark}

We end the introduction with two questions;

\begin{question}
    Is there a complete classification of clasp-closures of $3$-braids in the style of Birman-Menasco's classification of braid-closures of $3$-braids?
\end{question}
 Such a classification might allow one to obtain more classification results for links with categorified link invariants taking certain values.

\begin{question}
   Does annular Khovanov homology detect all clasp-closures of $3$-braids representing the unknot? Does Khovanov homology detect all of their augmentations? Does link Floer homology detect the links $\mathring{b}(\sigma_1\sigma_2^{n})$?
\end{question}
 
  \subsection*{Outline} 
  In Section~\ref{sec:HFL} we prove our results for link Floer homology. In Section~\ref{sec:Kh} we prove our Khovanov homology detection results. In Section~\ref{sec:AKH} we prove our annular Khovanov homology detection results as well as our two rank bounds.

 \subsection*{Acknowledgments} The author would like to thank Gage Martin for various helpful conversations. He would also like to thank Subhankar Dey for further helpful conversations as well as for providing careful feedback on an earlier draft of this paper. He is also grateful for~\cite{knotatlas} and~\cite{linkinfo}, which he found to be very helpful throughout the course of this project. {Finally, he would like to thank the referee for their careful reading of this paper.}

\section{Link Floer Homology}\label{sec:HFL}

In this section we collect our detection results for link Floer homology. In Section~\ref{sec:braidHFL} we show that link Floer homology detects $L6a2$ and $L9n15$. In Section~\ref{sec:claspbraidHFL} we {give a partial classification of} links with the link Floer homology types of augmentations of clasp-closures of index $3$-braids that represent the unknot.

\subsection{A review}

{Link Floer homology is an invariant due to Oszv\'ath-Szab\'o}~\cite{HolomorphicdiskslinkinvariantsandthemultivariableAlexanderpolynomial}{. It assigns to each oriented $n$-component link a finitely generated vector space equipped with $n+1$ gradings.  The first $n$ of these gradings are called \emph{Alexander gradings}, and the last is called the \emph{Maslov grading}. The Alexander gradings takes value in $\frac{1}{2}\Z$, while the Maslov grading takes value in $\Z$. For each component $K$ of a link $L$, there is a spectral sequences from $\widehat{\HFL}(L)$ to $\widehat{\HFL}(L\setminus K)\otimes V\Big[\frac{\lk(K,L\setminus L_i)}{2}\Big]$ corresponding to allowing pseudo-holomorphic disks to ``cross basepoints" in Heegaard diagrams}~\cite[Proposition 7.1]{HolomorphicdiskslinkinvariantsandthemultivariableAlexanderpolynomial}. {Throughout this paper we consider link Floer homology with $\Z/2$ coefficients unless explicitly stated otherwise.}

{Link Floer homology detects the Thurston norm, under mild hypotheses, by a result of Ozssv\'ath-Szab\'o}~\cite{ozsvath2008linkFloerThurstonnorm}{. It also detects braid closures by a result of Martin}~\cite[Proposition 1]{martin2022khovanov}.

\subsection{Braid-closures}\label{sec:braidHFL} We prove the following result:

\begin{theorem}\label{thm:HFLL6a2}
    Link Floer homology detects L6a2.
\end{theorem}

For the readers convenience we recall that the link Floer homology of $L6a2$ is given as follows;

  \begin{equation}
\centering
    \begin{tabular}{|c|c|c|c|c|}
    \hline
    \backslashbox{\!$A_2$\!}{\!$A_1$\!}
     & $\frac{-3}{2}$& $\frac{-1}{2}$&$\frac{1}{2}$&$\frac{3}{2}$ \\
\hline
$\frac{3}{2}$&&$\F_{-1}$&$\F_0$&\\\hline
$\frac{1}{2}$&$\F_{-3}$&$\F_{-2}^3$&$\F_{-1}^3$&$\F_0$ \\\hline
$\frac{-1}{2}$&$\F_{-4}$&$\F_{-3}^3$&$\F_{-2}^3$&$\F_{-1}$\\\hline
$\frac{-3}{2}$&&$\F_{-4}$&$\F_{-3}$& \\\hline
  \end{tabular}\label{table:HFLL6a2} 
\end{equation}

This can be deduced from, say, the fact that $L6a2$ is alternating, the multi-variable Alexander polynomial of $L6a2$, the signature of $L6a2$, and an application of~\cite[Theorem 1.3]{ozsvath2008holomorphic}.

For the proof, our strategy is to argue that if a link has the link Floer homology type as $L6a2$ then it is the augmentation of a braid-closure of a $3$-braid by applying Martin's braid axes detection result~\cite[Proposition 1]{martin2022khovanov}. We then appeal to Murasugi's classification of $3$-braids whose braid-closures are unknot{t}ed and note that link Floer homology distinguishes the corresponding links.

\begin{proof}[Proof of Theorem~\ref{thm:HFLL6a2}]
    Suppose $L$ is a link with the link Floer homology of L6a2. Observe that $L$ cannot be split, since its link Floer homology is not of the correct form. {More specifically, if $L$ were split then it could be written as $K_1\#(U\sqcup K_2)$ for $K_i$ knots and $U$ the unknot, where the connect sum is between $U$ and $K_1$. Consequently, the K\"unneth formula for link Floer homology,}~\cite[Theorem 1.4]{HolomorphicdiskslinkinvariantsandthemultivariableAlexanderpolynomial}, {would imply that each Alexander bigrading would have rank at least two, which we can observe is not the case from Table}~(\ref{table:HFLL6a2}). Since the rank of $\widehat{\HFL}(L)$ in the maximal non-trivial $A_1$ grading is two, it follows from~\cite[Proposition 1]{martin2022khovanov} that the first component of $L$, $L_1$, is a braid axis. Observe that the Conway polynomial of two{-}component links --- and hence knot Floer homology and link Floer homology --- detects the linking number of two{-}component links~\cite{hoste1985firstcoefficientoftheconwaypolynomial}. It follows that $L$ is the augmentation of the braid-closure of a $3$-braid, $L_2$.

    Let --- for the remainder of this section --- $V$ be a rank two vector space supported in Alexander grading $0$ and Maslov gradings $0$ and $-1$, and let $[a]$ indicate a shift in Alexander grading by $a$. There are spectral sequences from $\widehat{\HFL}(L)$ to $\widehat{\HFL}(L_i)\otimes V\Big[\frac{\lk(L)}{2}\Big]$ for each $i$ {corresponding to allowing pseudo-holomorphic disks to ``cross basepoints" in Heegaard diagrams}~\cite[Proposition 7.1]{HolomorphicdiskslinkinvariantsandthemultivariableAlexanderpolynomial}. We thus have that $\widehat{\HFL}(L_i)$ can be supported only in Alexander grading zero, {so that} each $L_i$ is the unknot. Thus{,} $L$ is the augmentation of a braid-closure of a $3$-braid representing the unknot. By Murasugi's classification of $3$-braids up to conjugacy there are exactly three $3$-braid representatives of the unknot; namely $(\sigma_1\sigma_2)^{\pm 1}$, and $\sigma_1\sigma_2^{-1}$~\cite{Murasugiclosed3}. Taking augmentations of the braid-closures of either of the first two braids yields $T(2,\pm 6)$, which have distinct link Floer homology from $L$. The result follows.
\end{proof}

\begin{remark}\label{rmk:l6a2orientations}
    Of course, a two{-}component unoriented link can --- apriori --- be endowed with four distinct orientations. However, $\overline{L6a2}$ is isotopic to the link obtained from $L6a2$ by reversing the orientation of either component. Likewise, the reverse of $L6a2$ is isotopic to $L6a2$. Thus, Theorem~\ref{thm:HFLL6a2} holds as a result for oriented links.
\end{remark}

We proceed to our next detection result.

\begin{theorem}\label{thm:HFLl9n15}
    Link Floer homology with rational coefficients detects L9n15.
\end{theorem}
We will not compute the link Floer homology of $\mathring{b}(\sigma_1^3\sigma_2)$. Instead{,} we will rely on formal properties of link Floer homology. The reason we take rational coefficients is that we will use the Khovanov homology of $L9n15$ to obtain information about link Floer homology via Dowlin's spectral sequence~\cite{dowlin2018spectral}, which is defined over the rational numbers. 
\begin{proof}
  We first study the link Floer homology of $\widehat{\HFL}(\mathring{b}(\sigma_1^3\sigma_2);\Q)$. Observe that, perhaps after relabeling components, $\widehat{\HFL}(\mathring{b}(\sigma_1^3\sigma_2);\Q)$ has maximal $A_2$ grading $\dfrac{3}{2}$, since we may take the second component to be the braid axis for the braid-closure $\mathring{b}(\sigma_1^2\sigma_2)$. Now, $\mathring{b}(\sigma_1^2\sigma_2)$ bounds a $3$-punctured torus, so we have that the maximal $A_1$-grading is at most $\frac{5}{2}$. In fact{,} the maximal $A_1$ grading must be at least $\frac{5}{2}$ since $\widehat{\HFL}(L;\Q)$ admits a spectral sequence to $\widehat{\HFL}(T(2,-3);\Q)\otimes V[-\frac{3}{2}]$, so that $\widehat{\HFL}(L;\Q)$ must have generators of {$A_{1}$} grading $\pm\frac{5}{2}$. From Knot atlas~\cite{knotatlas} we have that $\rank(\Kh(L9n15;\Z/2))=12$ --- see also Table~\ref{tab:Khl9n15} --- so that $6=\rank(\Khr(L9n15;\Z/2))\geq \rank(\Khr(L9n15;\Q))$ by the universal coefficient theorem and~\cite[Corollary 3.2.C]{shumakovitch2014torsion}. It follows that $\rank(\widehat{\HFL}(L9n15;\Q))\leq 12$ by an application of the rank bound from Dowlin's spectral sequence~\cite{dowlin2018spectral} together with some properties of pointed Khovanov homology~\cite[lemma 2.11]{baldwin2017khovanovpointed}. 

     Suppose $L$ is a link with the same link Floer homology with rational coefficients as $\mathring{b}(\sigma_1^3\sigma_2)$. Since $\widehat{\HFL}(L;\Q)$ determines the Conway polynomial of $L$ and the Conway polynomial of $L$ determines the linking number of two{-}component links~\cite{hoste1985firstcoefficientoftheconwaypolynomial}, it follows that $L$ has linking number $-3$. In particular $L$ is non-split. {Now,} the link Floer polytope {of $L$ agrees with that of $L9n15$. Since the link Floer polytope detects the Thurston polytope}~\cite{ozsvath2008linkFloerThurstonnorm} {for non-split links, it follows that the Thurston polytopes of $L$ and $L9n15$ agree. Consequently} $L_2$ bounds a surface in the exterior of $L$ of Euler characteristic $-2${, just as does the corresponding component in}  $Ln15$. Since such a surface necessarily has at least four {boundary components}, it follows that it is, in fact, a $4$-punctured disk, so that $L_2$ is an unknot. Since the rank in the maximum non-trivial Alexander grading is two, it follows that $L_2$ is a braid axis by~\cite[Proposition 1]{martin2022khovanov}.
     
     We now study the first component of $L$. From the link Floer polytope of $L$ we can see that $L_1$ bounds a surface in the exterior of $L$ of Euler characteristic $-4$. Since the linking number of $L$ is three, it follows that $L_1$ has Seifert-genus at most one. {In particular, }$\widehat{\HFL}(L_1;\Q)$ {is non trivial in Alexander gradings $\pm1$ and trivial in Alexander gradings $i$ with $|i|>1$. Since the rank of $\widehat{\HFL}(L_1;\Q)$ in Alexander gradings $\pm1$ are equal and the total rank of $\widehat{\HFL}(L;\Q)$ must be odd, it follows that} $\widehat{\HFL}(L_1;\Q)$ { must be non-trivial in Alexander grading zero.} We now prove that $L_1$ is fibered. Recall that there is a spectral sequence from $\widehat{\HFL}(L;\Q)$ {to} ${\widehat{\HFL}(L_1\otimes) V[-\frac{3}{2}]}$~\cite[Theorem 1.4]{HolomorphicdiskslinkinvariantsandthemultivariableAlexanderpolynomial}. {Observe that } ${\widehat{\HFL}(L)_1\otimes V[-\frac{3}{2}]}${ --- and so, in turn,  $\widehat{\HFL}(L;A_1=k;\Q)$ --- } must be of rank at least two for {$k=-\frac{1}{2},-\frac{3}{2},-\frac{5}{2}$}. In fact, by symmetry properties of link Floer homology{,} $\widehat{\HFL}(L,A_1=k;\Q)$ must be of rank at least two for {$k=\frac{1}{2},\frac{3}{2},\frac{5}{2}$} too. It follows that $\rank(\widehat{\HFL}(L;\Q))= 12$ and indeed that $\widehat{\HFL}(L;\Q)$ is of rank two in the maximal non-trivial $A_1$ grading. Martin's braid axis detection result implies that $L_1$ is a braid axis for $L_2$ and so, in particular, fibered~\cite[Proposition 1]{martin2022khovanov}. Indeed, since the maximal non-trivial $A_1$ grading is $1+\frac{3}{2}$, $L_1$ must be a genus one fibered knot. It follows that $L_1$ is a trefoil or a figure eight knot. To see that $L_1$ is a left handed trefoil, observe that $\widehat{\HFL}(L9n15,A_2=-\frac{5}{2};\Q)$ must be supported in Maslov gradings $0$ and $-1$ since it has a left-handed trefoil component and the linking number is $-3$. This in turn implies that $L_1$ must have a left handed trefoil component, since in the right handed trefoil case and Figure eight case $\widehat{\HFL}(L,A_1=-\frac{5}{2};\Q)$ would have to be supported in Maslov gradings $-2$ and $-3$ or $-1$ and $-2$ respectively.
     
     Now by Birman-Menasco's classification theorem for $3$-braids~\cite{BirmanMenasco3braids}, there are exactly two $3$-braids with braid-closures representing $T(2,-3)$, namely $\sigma_1^{-1}\sigma_2^{-3}$ and $\sigma_1\sigma_2^{-3}$, which is $L9n16$. These two links are distinguished by their Alexander polynomials, so the result follows.
\end{proof}

\begin{remark}\label{rem:L9n15orientations}
    Once again there are --- apriori --- four possible orientations with which L9n15 can be endowed. One pair of these have linking number $-3$ while the other has linking number $3$. It can be checked that each pair with the same linking number are, in fact, isotopic as links. That is, Theorem~\ref{thm:HFLl9n15} holds as a statement for oriented links.
\end{remark}

\subsection{Clasp-closures}~\label{sec:claspbraidHFL} In this section we {study} the links with the link Floer homology type of augmentations of clasp-closures of $3$-braids representing the unknot. In particular, we obtain the following two theorems advertised in the introduction:

\begin{theorem}\label{thm:l7a6plus}
   { Link Floer homology detects the Mazur link.}
\end{theorem}

\begin{theorem}\label{thm:infinitefamily}
    Let $L$ be a link. $\widehat{\HFL}(L)\cong\widehat{\HFL}(\mathring{c}(\sigma_2^{-1}\sigma_1\sigma_2))$ if and only if $L$ is of the form $\mathring{c}(\sigma_1^n\sigma_2^{-1}\sigma_1\sigma_2)$ for some $n\in\Z$.
\end{theorem}

We begin by discussing some structural properties of the link Floer homology of links that are augmentations of clasp-closures of $3$-braids representing the unknot. Let $L$ be such a link, with the first component of $L$, $L_1$, being the clasp-closure of the $3$-braid and the second component, $L_2$, its axis.  The maximal $A_2$ grading in which $\widehat{\HFL}(L,A_2)$ has non-trivial support is $\frac{3}{2}$. This follows from~\cite[Theorem 1.1]{ozsvath2008linkFloerThurstonnorm}.

We also have the following result:

\begin{lemma}\label{lem:maxgradingmasloov}
    Suppose $L$ is the augmentation of a clasp-closure, $L_1$, of braid representing a knot. Then the component of $\widehat{\HFL}(L)$ with maximal non-trivial $A_2$ grading is given by $\F_{-1}[-1]\oplus\F_{0}^2[0]\oplus\F_1[1]$ up to overall shifts in the Maslov and $A_1$ gradings.
\end{lemma}

A version of this result without the Maslov grading is given in~\cite[Lemma 5.9]{binns2024floer}. The proof of this Lemma requires techniques from sutured Floer homology. The reader is directed to Juh\'asz' papers~\cite{juhasz2006holomorphic,juhasz2008floer,juhasz2010sutured} for necessary background.

\begin{proof}
Suppose $L$ is as in the statement of the Lemma. A sutured Heegaard diagram for the sutured manifold $(Y,\gamma)$ obtained by decomposing the exterior of $L$ along an appropriate maximal Euler characteristic longitudinal surface for $L_2$ is shown in Figure~\ref{fig:suturedhd}.

 \begin{figure}[h]
 \centering
 \includegraphics[width=5cm]{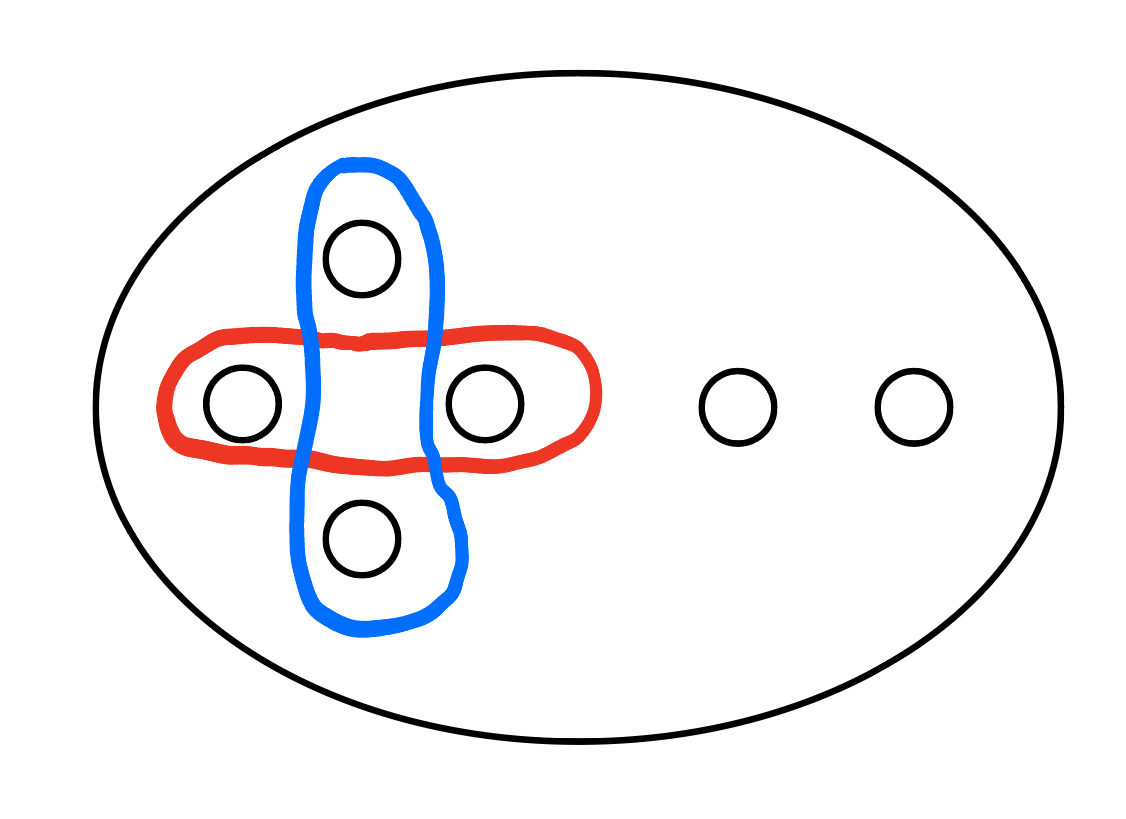}
    \caption{A sutured Heegaard diagram for the sutured manifold obtained by decomposing an augmented clasp-closure along a longitudinal surface for its axis. The outer boundary component of the surface corresponds to a longitude of $L_2$, while the inner boundary components correspond to meridians of $L_1$.}\label{fig:suturedhd}
\end{figure}

    A-priori $\SFH(Y,\gamma)$ only comes with a relative Maslov grading in each $\spin^c$ structure. However, in the case at hand these Maslov gradings can be upgraded to a relative Maslov grading that applies across all $\spin^c$ structures. To see this{,} observe that capping off the sutures corresponding to meridians of $L_1$ results in another sutured manifold, $(\widehat{Y},\widehat{\gamma})$ in which all four of the generators of $\CF(\widehat{Y},\widehat{\gamma})$ are supported in a single $\spin^c${-}structure. The claim follows. It remains to check that the maps $\hat{p}$ from~\cite[Proposition 5.4]{juhasz2010sutured} respect this relative Maslov grading which applies across all $\spin^c$ structures. However, this follows by repeating Juh\'asz' proof of~\cite[Proposition 5.4]{juhasz2010sutured}. Specifically, there is a Heegaard diagram for $(Y,\gamma)$ that can be obtained by doubling a Heegaard diagram for the exterior of $L$ along a certain subsurface~\cite[Proposition 5.2]{juhasz2008floer}, and pseudo-holomorphic disks from the doubled Heegaard diagram correspond to disks in the Heegaard diagram for $L$~\cite[proposition 7.6]{juhasz2008floer}. This correspondence still holds if we fill in boundary components, yielding the desired result.
\end{proof}

We can now prove that link Floer homology detects augmented clasp-closures of $3$-braids. Our proof depends on the considerably more general classification of links with link Floer homology of next to minimal rank in the maximal non-trivial Alexander grading of a given component due to the author and Dey~\cite[Theorem 5.1]{binns2024floer}.

\begin{lemma}\label{lem:3clasprepunknot}
   Suppose that a link $L$ has the link Floer homology type of an augmented clasp-closure of a $3$-braid representing the unknot. Then $L$ is an augmented clasp-closure of a $3$-braid.
\end{lemma}

\begin{proof}
   Suppose $L$ is as in the statement of the Lemma. We first claim that $L$ has linking number $\pm1$. Note that augmented clasp-closure of $3$ braids have linking number $\pm1$. Now recall that the Conway polynomial --- and hence link Floer homology --- detects the linking number of{ }two{-}component link{s}~\cite{hoste1985firstcoefficientoftheconwaypolynomial}. The claim follows.
   
   Now, after relabeling the components of $L$ if necessary, we may assume that the component of $\widehat{\HFL}(L)$ with maximal non-trivial Alexander grading of rank four is $L_2$ and that the maximal non-trivial $A_2$ grading is $\frac{3}{2}$ and that $\widehat{\HFL}(L,A_2=\frac{3}{2})$ is given by ${\F_{-1}[-1]\oplus\F_{0}^2[0]\oplus\F_1[1]}$, up to shifts in the $A_1$ and Maslov gradings by Lemma~\ref{lem:maxgradingmasloov}. We now bound the genus of the component $L_2$. Recall that there is a spectral sequence from $\widehat{\HFL}(L)$ to $\widehat{\HFL}(L_2)\otimes V[\pm\frac{1}{2}]${. I}t follows that the maximum non-trivial Alexander grading in which $L_1$ can have non-trivial support is at most one.
   
  By~\cite[Theorem 5.1]{binns2024floer} we have four cases to treat:\begin{enumerate}

       \item  $L_2$ is a genus one fibered knot and $L_1$ is a clasp-braid with axis $L_2$.
       \item $L_2$ is a genus one nearly fibered knot and $L_1$ is a braid-closure with axis $L_2$.
       \item $L_2$ is a fibered knot and $L_1$ can be isotoped to a simple closed curve in a minimal genus Seifert surface for $L_2$.
      
       \item $L_1$ is a clasp-closure with $L_2$ its unknotted axis .
       
  \end{enumerate}

  For definitions of ``nearly fibered" see~\cite{baldwin2022floer}. For a definition of what it is to be braided with respect to a nearly fibered knot see~\cite[Section 3]{binns2024floer}. We rule out the first three of the four possibilities.
  
  For the first case observe that the maximal Euler characteristic of a longitudinal surface for $L_2$ would be $-3$, so that the maximal $A_2$ grading in which $\widehat{\HFL}(L)$ would be non-trivial support would be $\frac{5}{2}$ by~\cite[Theorem 1.1]{ozsvath2008linkFloerThurstonnorm}, a contradiction.

For the second, recall that there is a spectral sequence from $\widehat{\HFL}(L)$ to $\widehat{\HFK}(L_2)\otimes V[\pm\frac{1}{ 2}]$. Since in the maximal non-trivial $A_2$ grading $\widehat{\HFL}(L)$ is of rank four, as is the rank of the maximal non-trivial Alexander grading of $\widehat{\HFK}(L_2)\otimes V[\pm\frac{1}{ 2}]$, it follows that this spectral sequence collapses immediately. In particular{,} it follows that the component of $\widehat{\HFK}(L_2)\otimes V$ in maximal non-trivial Alexander grading is given up to an overall shift in Maslov grading by $\F_{-1}\oplus\F_{0}^2\oplus\F_{1}$. Now, Baldwin-Sivek classified all genus one nearly fibered knots~\cite{baldwin2022floer}.{All} such knots have the property that their knot Floer homology in Alexander grading one is supported in exactly one Maslov grading --- see~\cite[Table 1]{baldwin2022floer} --- a contradiction.

The third case is immediately excluded by the fact that the linking number of $L$ is non-zero, {concluding the proof}.
\end{proof}

\begin{lemma}\label{lem:extra}
  {Suppose $L$ is a link with the link Floer homology of either the Mazur link or $\mathring{c}(\sigma_2^{-1}\sigma_1\sigma_2)$. Then $L$ has unknotted components each of which is an unknotted clasp closure with respevct to the other.}
\end{lemma}

\begin{proof} {Suppose $L$ is as in the statement of the Lemma. First observe that the for both the Mazur link and $\mathring{c}(\sigma_2^{-1}\sigma_1\sigma_2)$, each component is a clasp-closure of a $3$-braid with respect to the other. The claim now follows from the same argument as given in the previous lemma, but now applied to both components of $L$.}
\end{proof}

By the preceding lemma{s}, to complete the proofs of Theorem~\ref{thm:l7a6plus} and Theorem~\ref{thm:infinitefamily} it suffices to {show that} link Floer homologies {distinguished between appropriate} augmentations of clasp-closures of $3$-braids representing the unknot.

We first address the links corresponding to the infinite family of braids $\sigma^n\sigma_2^{-1}\sigma_1\sigma_2$. It can be checked that the unoriented resolution of this link at the crossing shown in Figure~\ref{fig:inffamilyres} results is the split sum of a Hopf link and an unknot. Recall that J.Wang showed that if $L_b$ is a band sum of the split union of two links $L_1\sqcup L_2$ then the link Floer homology of {$L_b$} does not change after adding twists to the band~\cite[Remark 1.18]{Wangcosmetic}. Thus the links $\mathring{c}(\sigma^n\sigma_2^{-1}\sigma_1\sigma_2)$ all have the same link Floer homology.

\begin{center}
 \begin{figure}[h]
 \includegraphics[width=5cm]{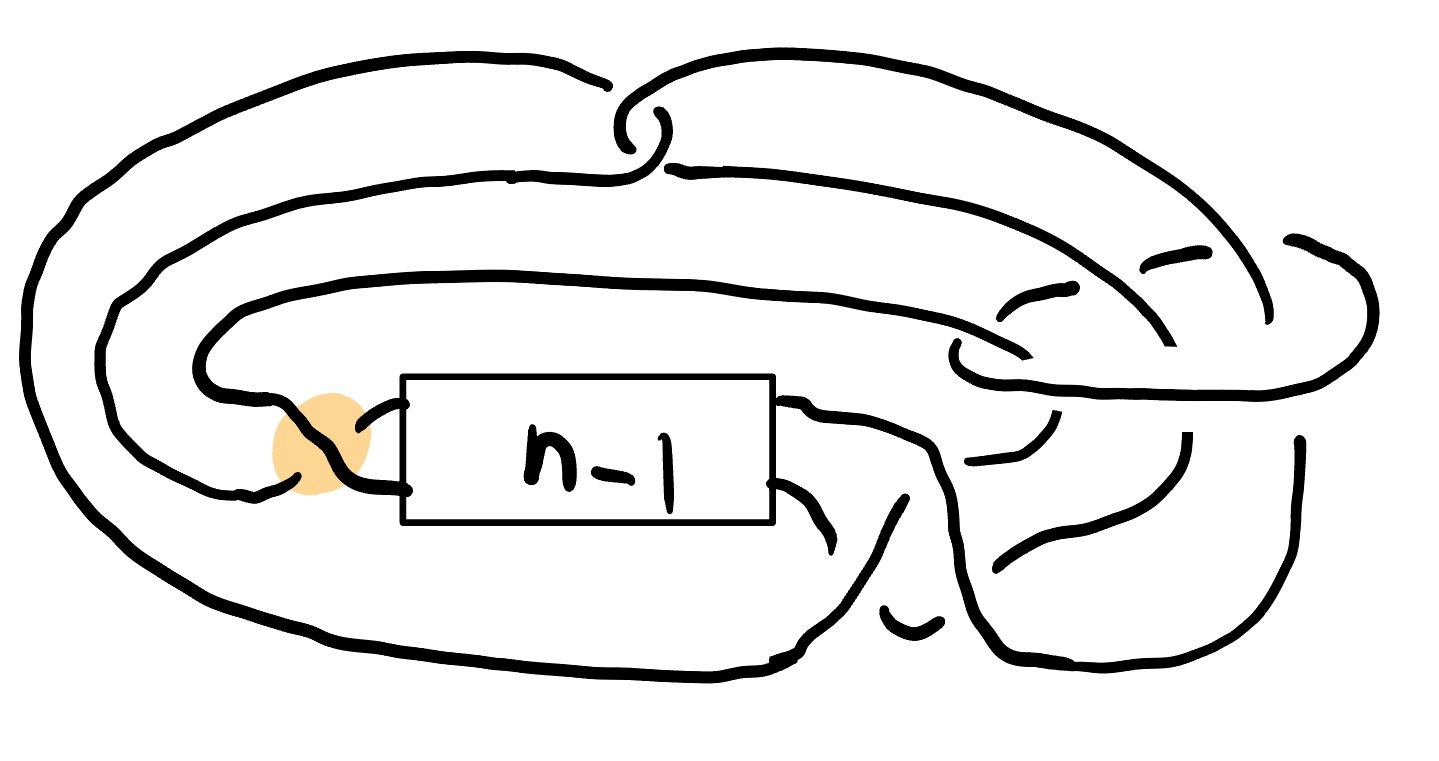}
    \caption{The link $\mathring{c}(\sigma^n_1\sigma_2^{-1}\sigma_1\sigma_2)$. We consider the unoriented resolution of the crossing highlighted in orange.}\label{fig:inffamilyres}
   \end{figure}
\end{center}

We can now conclude the proofs of two of the results promised in the introduction.

\begin{proof}[Proof of Theorem~\ref{thm:l7a6plus} and Theorem~\ref{thm:infinitefamily}]

{Suppose that $L$ is a link as in one the two Theorem statements.} By Lemma~\ref{lem:3clasprepunknot} {and Lemma}~\ref{lem:extra}, {both components of $L$ are clasp-closures of index $3$-braids with respect to the other component. In particular, the maximum $A_i$ gradings in which $\widehat{\HFL}(L)$ are non-trivial are $\frac{3}{2}$. On the other hand, $\mathring{c}(\sigma_1^3\sigma_2^{-1}\sigma_1^2\sigma_2)$ does not have this property; one of the components --- say $L_1$ --- is the clasp-closure of an index $5$ braid. Consequently, the maximum $A_1$ grading in which $\widehat{\HFL}(\mathring{c}(\sigma_1^3\sigma_2^{-1}\sigma_1^2\sigma_2))$ is non-trivial is $\frac{5}{2}$. It follows that $L$ cannot be $\mathring{c}(\sigma_1^3\sigma_2^{-1}\sigma_1^2\sigma_2)$.}  By~\cite[Remark 1.18]{Wangcosmetic}, ${\widehat{\HFL}(\mathring{c}(\sigma_1^n\sigma_2^{-1}\sigma_1\sigma_2))\cong \widehat{\HFL}(\mathring{c}(\sigma_2^{-1}\sigma_1\sigma_2))}$ for all $n$. {It thus} suffices to show that no two of the Mazur link, 
 $\mathring{c}(\sigma_2^{-1}\sigma_1\sigma_2)$, and their mirrors have the same link Floer homology.

   Now, given that link Floer homology detects the linking number of two{-}component links, the transformation property of link Floer homology under changing the orientation of a link component implies that for a two{-}component link, $L$, $\widehat{\HFL}(L)$ determines the Alexander polynomial of $L$ endowed with an arbitrary orientation. {For the two links at hand, these} are given as follows:
\begin{enumerate}
  
       \item $\Delta(\mathring{c}(\sigma_1^{-1}))=2-5t+5t^2-2t^3$.
       \item $\Delta(\mathring{c}'(\sigma_1^{-1}))=1-3t+3t^2-3t^3+3t^4-t^5$.

       \item $\Delta(\mathring{c}(\sigma_2^{-1}\sigma_1\sigma_2))=2-7t+7t^2-2t^3$,
       \item $\Delta(\mathring{c}'(\sigma_2^{-1}\sigma_1\sigma_2))=1-3t+5t^2-5t^3+3t^4-t^5$.

\end{enumerate}
      Here the primes indicate that the orientation of the axis has been reversed.  The author used~\cite{KLO} for these computations. {Note that the Alexander polynomial of a link is unchanged under mirroring. Thus link Floer homology distinguishes the Mazur link and its mirror, from
 $\mathring{c}(\sigma_2^{-1}\sigma_1\sigma_2)$ and its mirror. Since the Mazur link and its mirror have linking numbers $\pm1$ --- and likewise for $\mathring{c}(\sigma_2^{-1}\sigma_1\sigma_2)$ --- the Alexander polynomials distinguishes each pair of links. Consequently link Floer homology distinguishes between the four links, concluding the proof.}
\end{proof}

\begin{remark}
    The Mazur link and its reverse are isotopic, so for this link we have oriented link detection on the nose.
\end{remark}

\section{Khovanov Homology}\label{sec:Kh}

The goal of this section is to prove the following two results:

\begin{theorem}\label{thm:khl6a2}
    Khovanov homology  with integer coefficients detects $L6a2$.
\end{theorem}

\begin{theorem}\label{thm:KhL9n15}
    Khovanov homology with integer coefficients detects $L9n15$.
\end{theorem}

Given Martin's result that Khovanov homology detects $T(2,6)$ oriented as a $2$-braid closure~\cite{martin2022khovanov} --- from which it can be deduced that Khovanov homology detects $T(2,\pm 6)$ with both orientations --- we have that Khovanov homology detects all augmentations of braid-closures of $3$-braids representing the unknot.

\subsection{A review}
{We begin with a brief review of Khovanov homology. \emph{Khovanov homology} is an invariant of oriented links due to Khovanov}~\cite{khovanov2000categorification}{. It assigns to each link a $\Z$-module equipped with two gradings; the first called the \emph{quantum grading} and the second called the \emph{homological grading}. The Khovanov homology of an $n$-component link is supported in quantum gradings whose parity agrees with that of $n$. Taking coefficients in a field, the Khovanov homology of $L$ admits a link splitting spectral sequence to the Khovanov homology of the disjoint union of the underlying components of $L$. This is due to Batson-Seed}~\cite{batson2015link}{. Taking coefficients in $\Q$, Khovanov homology admits a spectral sequence to an invariant called \emph{Lee homology}, a result of}~\cite{lee2005endomorphism}{. Lee homology is a finitely generated $\Q$-vector space equipped with a homology grading compatible with that of Khovanov homology under the Lee spectral sequence. For a two component link $L$, the Lee homology of $L$ is of rank four,  with rank two in homological grading zero and rank four in homological grading the linking number of $L$.

There is a version of Khovanov homology called \emph{pointed Khovanov homology} due to Baldwin-Levine-Sarkar}~\cite{baldwin2017khovanovpointed}{. Taking coefficients in $\Q$, pointed Khovanov homology admits a spectral sequence to knot Floer homology, a result due to Dowlin}~\cite{dowlin2018spectral}{. Pointed Khovanov homology is a generalization of an invariant called \emph{reduced Khovanov homology.} Taking coefficients in $\Z/2$, the rank of reduced Khovanov homology is exactly half that of Khovanov homology, a result of Shumakovitch}~\cite[Corollary 3.2.C]{shumakovitch2014torsion}.
\subsection{Detection results}
For the reader's convenience we note that the Khovanov homology of L6a2 is given by;

  \begin{center}

    \begin{tabular}{|c|c|c|c|c|c|c|c|}
    \hline
    \backslashbox{\!$q$\!}{\!$h$\!}
     & $-6$& $-5$&$-4$&$-3$ &$-2$&$-1$& $0$  \\
\hline
$-2$& & & &  & & &$\Z$\\\hline
$-4$& & & &  & &$\Z$&$\Z$\\\hline
$-6$& & & &  &$\Z\oplus\Z/2$& & \\\hline
$-8$& & &  &$\Z\oplus\Z/2$&$\Z$& & \\\hline
$-10$& & &$\Z\oplus\Z/2$&$\Z$& & &  \\\hline
$-12$& &$\Z/2$&$\Z$&  & & & \\\hline
$-14$&$\Z$&$\Z$& &  & & & \\\hline
$-16$&$\Z$& &  & & & & \\\hline
  \end{tabular} 
\end{center}

See~\cite{knotatlas}.

\begin{lemma}\label{lem:l6a2unknottedlinking}
    Suppose $L$ is a link with the Khovanov homology of the $L6a2$. Then $L$ is a two{-}component link with each component an unknot. Moreover $\lk(L)=-3$.
\end{lemma}

\begin{proof}
    Suppose $L$ is as in the statement of the theorem. Observe that $L$ has an even number of components since the quantum grading is supported in even gradings. The Batson-Seed link splitting spectral sequence~\cite{batson2015link}, together with an application of the universal coeffic{ie}nt theorem implies that \begin{align}\rank(\Kh(L;\Z/2))=20\geq\prod\rank(\Kh(L_i;\Z/2)).\label{eq:rankh}\end{align} Here{,} the product is taken over components $L_i$ of $L$. Since $\rank(\Kh(L_i;\Z/2))$ is of the form $2+4k_i$ for some $k_i\geq 0$, we must have that $L$ has at most four components. If $L$ has exactly four components{,} then each component is an un{knot} by~\cite{kronheimer2011khovanov}.  To check that this is impossible, we use the refined version of the Batson-Seed link splitting spectral sequence. Equip $\Kh(L;\Q)$ with the the $l:=h-q$ grading{. T}hen~\cite[Corollary 4.4]{batson2015link} implies that for some constant $t$ \begin{align}
        \rank^{l}(\Kh(L;\Z/2))\geq \rank^{l+t}(W^{\otimes4}).
    \end{align}

   Here, and for the remainder of this proof, $W$ is the rank two vector space supported in $l$ gradings $1$ and $-1$. In particular{,} there is some $l$ grading in which ${\rank(\Kh^l(L;\Z/2))\geq 6}$. This is false by inspection. Thus{,} $L$ has exactly two components. Observe that Equation~(\ref{eq:rankh}) implies that at least one component of $L$ is an unknot, since the unknot is the unique knot $K$ with $\rank(\Kh(K;\Z/2))=2$~\cite{kronheimer2011khovanov}. The remaining component $K_2$ of $L$ has ${\rank(\Kh(K_2;\Z/2))\leq 6}$, so that it is either an unknot or a trefoil by~\cite{kronheimer2011khovanov} and~\cite{baldwin2021khovanov}. 
    
    To see that $\lk(L)=-3$, recall that $\Kh(L;\Q)$ --- which can be obtained from  $\Kh(L;\Z)$ by an application of the universal coefficient theorem --- admits a spectral sequence to Lee homology~\cite{lee2005endomorphism}. Lee homology carries a homological grading and the spectral sequence respects this grading. Moreover, the Lee homology of a two{-}component link $L$ is supported in homological gradings $0$ and $2\lk(L)$, and that each such grading contributes a $\Q^2$ summand. By inspection of $\Kh(L;\Q)$ we see that $L$ must have linking number $-3$.
    
    To check that the remaining component of $L$ is also unknotted, we use the refined version of the Batson-Seed link splitting spectral sequence again. Equip $\Kh(L;\Q)$ with the the $l:=h-q$ grading then~\cite[Corollary 4.4]{batson2015link} implies that \begin{align}\label{eq:BSforlk=-3}
        \rank^l(\Kh(L;\Q))\geq \rank^{l-3}(\Kh(K_2;\Q)\otimes W),
    \end{align}

    \noindent where $W$ is the rank two vector space supported in $l$ gradings $1$ and $-1$.   Since $\Kh(T(2,3);\Q)$ has a generator in $l$ grading $-6$,  $\Kh(K_2;\Q)\otimes W$ has a generator of $l$ grading $-7$, violating the rank bound. Likewise since $\Kh(T(2,-3);\Q)$ has two generators in $l$ grading $3$,  $\Kh(T(2,-3);\Q)\otimes W$ has two generators in $l$ grading $2$ violating the rank bound. \end{proof}

The remainder of the proof of Theorem~\ref{thm:khl6a2} amounts to showing that there is a component of $L$ that is a braid axis for the other. To do so{,} we use Dowlin's spectral sequence from an appropriate version of Khovanov homology to $\widehat{\HFK}(L;\Q)$ to reduce this question to a question about link Floer homology~\cite{dowlin2018spectral}.
\begin{proof}[Proof of Theorem~\ref{thm:khl6a2}]
    Suppose $L$ is as in the statement of the Theorem. By the previous Lemma, $L$ has two components. Since $L$ has $\delta${-}thin Khovanov homology, we have that $\widehat{\HFL}(L;\Q)$ is $\delta$-thin. A result of the author and Dey~\cite[Proposition 6.1]{binns2022rank} implies in turn that $\widehat{\HFL}(L;\Q)$ decomposes as a direct sum of vector spaces of the form \begin{align*}W_a[b,c]:=\Q_{a-1}[b-\frac{1}{2},c-\frac{1}{2}]\oplus\Q_{a}[b+\frac{1}{2},c-\frac{1}{2}]\oplus \Q_{a}[b-\frac{1}{2},c+\frac{1}{2}]\oplus\Q_{a+1}[b+\frac{1}{2},c+\frac{1}{2}].\end{align*} Here $\Q_a[b,c]$ is a $\Q$ summand in $(A_1,A_2)$ grading $(b,c)$ of Maslov grading $a$. There are at most five of these summands since $20=\rank(\Kh(L;\Z/2))\geq \rank(\widehat{\HFK}(L;\Q))$, as follows form Dowlin's spectral sequence~\cite{dowlin2018spectral} together with the same steps applied in the corresponding stage of the proof of Theorem~\ref{thm:HFLl9n15}.
    
 Observe that if the span of an
    $A_i$ grading is $[\frac{-1}{2},\frac{1}{2}]$ then $L_i$ is a meridian of the other component, so that $L$ is a Hopf link{, since each component is unknotted by Lemma}~\ref{lem:l6a2unknottedlinking}. Neither Hopf link has the correct Khovanov homology, so the span of {each} Alexander gradings must be strictly larger {than $[\frac{-1}{2},\frac{1}{2}]$}.

     If {$\widehat{\HFL}(L;\Q)$ contains} an odd number of {$W_a[b,c]$} summands{,} the symmetry of link Floer homology implies that $\widehat{\HFL}(L;\Q)$ contains a {$W_b[0,0]$} summand{, for some $b$}. It follows that either:\begin{enumerate}
            \item $\widehat{\HFL}(L;\Q)$ has {$W_a[m,n]^{\oplus 2}\oplus W_{a-2m-2n}[-m,-n]^{\oplus 2}$} summand, where {$m,n\in\frac{1}{2}\Z, m,n\geq \frac{1}{2}$} and {$a=b+{n+m}$ if there is a {$W_b[0,0]$} summand},

            \item $\widehat{\HFL}(L;\Q)$ has a {$W_a[m,n]\oplus W_{a-2m}[-m,n]\oplus W_{-a-2m-2n}[-m,-n]\oplus  W_{-a-2n}[m,-n]$} summand where {$m,n\in\frac{1}{2}\Z, m,n\geq \frac{1}{2}$ and $a=b+{n+m}$ if there is a {$W_b[0,0]$} summand},
            \item or in a maximal non-trivial $A_i$ grading, $\widehat{\HFL}(L)$ is of rank two.
        \end{enumerate}

    Suppose {that} we are in one of the first two cases. $L$ is non-split because both components are unknotted and the Khovanov homology of the two{-}component unlink {i}s of rank four. Thus{,} we can apply~\cite[Theorem 5.1]{binns2024floer}. Since $L_1$ is unknotted, we deduce that $L_2$ is a clasp-closure with respect to $L_1$. However, Lemma~\ref{lem:maxgradingmasloov} then implies that the maximal non-trivial $A_1$ grading is given {---} up to affine isomorphism {---} by ${\Q[-1]\oplus\Q^2[0]\oplus\Q[1]}$. This is a direct contradiction in case $1$. In case two we would then have that $\widehat{\HFL}(L;\Q)$ contains a {$W_a[n,\frac{1}{2}]\oplus W_{a-1}[n,-\frac{1}{2}]\oplus W_{a-1-2n}[-n,-\frac{1}{2}]\oplus  W_{a-2n}[-n,\frac{1}{2}]$} summand, a contradiction since $L$ has odd linking number so that $\widehat{\HFL}(L;\Q)$ must be supported in $\Z+\frac{1}{2}$ valued Alexander gradings. 

    Thus we have that $\widehat{\HFL}(L;\Q)$ is of rank two in one of the maximal non-trivial Alexander gradings. By~\cite[Proposition 1]{martin2022khovanov} we have that one component --- say $L_1$ --- is a braid axis for the other --- say $L_2$. Since the linking number of the two links is $-3$, it follows that $L_2$ is the braid-closure of a $3$-braid. Since $L_2$ represents the unknot, the desired result follows from Murasugi's classification of $3$-braids with unknotted braid-closures up to conjugacy~\cite{Murasugiclosed3} and the fact that $T(2,\pm6)$ has distinct Khovanov homology from $L6a2$.
    \end{proof}

    \begin{remark}
        Using the same argument as given in Remark~\ref{rmk:l6a2orientations} it can be shown that Khovanov homology detects $L6a2$ regardless of the orientation.
    \end{remark}

We now proceed to our next detection result, for $L9n15$. For the reader's convenience we recall from~\cite{knotatlas}, say, that the Khovanov homology of $L9n15$ is given as follows:

\begin{align}
    \begin{tabular}{|c|c|c|c|c|c|c|c|}
    \hline
    \backslashbox{\!$q$\!}{\!$h$\!}
     & $-6$& $-5$&$-4$&$-3$ &$-2$&$-1$& $0$  \\\hline
$-6$& & & &  && & $\Z$ \\\hline
$-8$& & &  &&$\Z/2$& & $\Z$\\\hline
$-10$& & &&&$\Z$ & &  \\\hline
$-12$&&& $\Z$&   & & & \\\hline
$-14$&&& $\Z$& $\Z$  & & & \\\hline
$-16$&$\Z$& $\Z$&  & & & & \\\hline
$-18$&$\Z$& $\Z$ &  & & & & \\\hline
  \end{tabular}\label{tab:Khl9n15}
\end{align}

\begin{proof}[Proof of Theorem~\ref{thm:KhL9n15}]
    Suppose $L$ is a link with $\Kh(L;\Z)\cong\Kh(L9n15;\Z)$. We first determine the components of $L$. Note that $L$ has an even number of components since $\Kh(L;\Z)$ is supported in even quantum gradings. Observe that $\Kh(L;\Z/2)\cong\Kh(L9n15;\Z/2)$ by an application of the universal coefficient theorem. In particular, $\rank(\Kh(L;\Z/2))=12$. Consider the Batson-Seed link splitting spectral sequence~\cite{batson2015link}. Since every {knot} has Khovanov homology with $\Z/2$ coefficients of rank $2+4m$ for some $m$, we have that $L$ has at most two components. Indeed {--- appealing to}~\cite[Corollary 3.2.C]{shumakovitch2014torsion} --- one of these components, {without loss of generality} $L_1$, has ${2\rank(\Khr(L_1;\Q)\leq\rank(\Kh(L_1;\Z/2))=2}$ so that $L_1$ is unknotted by~\cite{kronheimer2011khovanov}. The remaining component of $L$, $L_2$, has  ${\rank(\Kh(L_1;\Z/2))\leq 6}$ and so{,} in turn{,} ${\rank(\Khr(L_1;\Z/2))\leq 3}$ by~\cite[Corollary 3.2.C]{shumakovitch2014torsion}. It follows from~\cite{baldwin2021khovanov} and~\cite{kronheimer2011khovanov} that $L_2$ is an unknot or a trefoil.

       An application of the universal coefficient theorem shows that $\Kh(L;\Q)\cong\Kh(L9n15;\Q)$. Consider the spectral sequence from Khovanov homology to Lee homology. Since this spectral sequence respects the homological grading and the Lee homology of a two{-}component link consists of two $\Q\oplus\Q$ summands supported in homological gradings $0$ and $2\lk(L)$, we can see by inspection that $\lk(L)=-3$. We can then apply {E}quation~(\ref{eq:BSforlk=-3}) again{; or rather to the same equation but where we take $\Z/2$-coefficients rather than rational coefficients}. Let $U$ denote the unknot. Note that $\Kh(U;\Z/2)\otimes W$ has support in {$l:=h-q$} grading $2$ so that $L_1$ cannot be the unknot. Likewise $\Kh(T(2,3);\Z/2)\otimes W$ has support in {$l:=h-q$} grading $-7$ so that in fact $L_2$ is $T(2,-3)$.

       Now, by an application of by~\cite[Corollary 3.2.C]{shumakovitch2014torsion} {and the universal coeficeint theorem,} we have that $\rank(\Khr(L;\Q))\leq 6$ so that $\rank(\widehat{\HFK}(L;\Q))\leq 12$ by the rank bound coming from Dowlin's spectral sequence~\cite{dowlin2018spectral}. Recall that there is a spectral sequence from $\widehat{\HFL}(L;\Q)$ to ${\widehat{\HFL}(L_2;\Q)\otimes V[-\frac{3}{2}]}$ Since $L$ has linking number $-3$ and the second component is a copy of $T(2,-3)$ it follows that in $A_2$ grading {$-\frac{5}{2}$} there is a $\Q_0\oplus\Q_{-1}$-summand, in $A_2$ grading {$-\frac{3}{2}$} there is a $\Q_1\oplus\Q_{0}$-summand and that in $A_2$ grading $-\frac{1}{2}$ there is a $\Q_1\oplus\Q_{2}$-summand. This completely determines the $A_2$-graded version of {the} link Floer homology of $L$ by symmetry properties and the fact that the rank is at most twelve. Now, since $L_1$ is unknotted and the linking number of $L$ is $-3$, $\widehat{\HFL}(L;\Q)$ must have support in Alexander gradings $\pm\frac{3}{2}$. Since homogeneous summands with $A_2$-grading at least $0$ must all die under the spectral sequence to $\widehat{\HFL}(T(2,-3);\Q)\otimes V[-\frac{3}{2}]$, we have that the pairs of generators in each $A_2$ grading must be of distinct $A_1$ gradings.

       Now, the span of $\delta$-graded $\Kh(L;\Z/2)$ is $4$, so the span of $\delta$-graded $\Khr(L;\Z/2)$ is $2$. Thus{,} the span of $\delta$-graded pointed Khovanov homology, $\widetilde{\Kh}(L,\mathbf{p};\Z/2)$ where $\mathbf{p}$ consists of a point on each component of $L$, is at most two~\cite[Lemma 2.11]{baldwin2017khovanovpointed} and so finally the span of {$\delta$-graded} $\widehat{\HFL}(L;\Q)$ is at most two. It follows that at most three of the homogeneous $\Q$ summands with $A_2$-grading at most $-\frac{1}{2}$ occur in extremal $A_1$ gradings. It follows in turn that $\widehat{\HFL}(L);\Q$ is of rank two in the maximal non-trivial $A_1$ grading, so that $U$ is a braid axis for $T(2,-3)$. Since the linking number is $-3$, the corresponding braid is a $3$-braid. Now, by Birman-Menasco's classification of $3$-braids with braid-closures representing the unknot, the only two such augmented braid-closures are $L9n15$ and $L9n16$. These are distinguished by their Khovanov homology, see~\cite{knotatlas}. The result holds for oriented links by Remark~\ref{rem:L9n15orientations}.\end{proof}

\section{Annular Khovanov homology}\label{sec:AKH}
In this section we study annular Khovanov homology. In Section~\ref{subsec:review} we review structural properties of the invariant we will use in the rest of the section. In Section~\ref{subsec:rankbound} we prove rank bounds for the annular Khovanov homology of clasp-closures and braid-closures and give some applications of the rank bounds to the study of braid-closures. In Section~\ref{sec:applications} we give two braid-closure detection results. In Section~\ref{sec:akhclasps} we apply a rank bound from Section~\ref{subsec:rankbound} to prove that annular Khovanov homology detects the Mazur pattern.

\subsection{A review}\label{subsec:review} 

We begin with a brief review of annular Khovanov homology{, which is due to Asaeda-Przytycki-Sikora}~\cite{asaeda_categorification_2004}. We will work with coefficients in $R$, where $R$ is either $\Z$, $\C$, $\Q$ or $\Z/2$. Annular Khovanov homology is an $R$-module valued invariant of links in the thickened annulus. The underlying chain complex for the annular Khovanov homology of an annular link $L$ is freely generated by complete resolutions of a fixed diagram for $L$ where each circle is decorated with a $1$ or an $X$. The resulting homology groups carry three gradings. The first of these gradings is called the \emph{homological grading} which we shall denote by $i$, the second is the \emph{quantum grading} which we shall denote by $j${. These two gradings are defined just as in the Khovanov homology case.} {T}he third grading is called the \emph{annular grading}{,} which we shall denote by $k$. {This is defined as the difference between the number of resolutions encircling the annular axis marked with a $1$ and those marked with an $X$. The differential on annular Khovanov homology is then simply defined as the components of the differential on the Khovanov complex of the underlying link that preserve the annular grading.}

We will use an two exact triangles for Annular Khovanov homology. Recall --- say from~\cite[Lemma 8.2]{binns2020knot} --- that annular Khovanov admits the following skein exact triangle corresponding to resolving a negative crossing:

\begin{equation}\label{eq:akhskeinneg}
\begin{tikzcd}
    \AKh(L)\arrow[rr]&&\AKh(L_0)[n_-^0-n_-]\{3n_-^0-3n_-+1\}\arrow[dl]\\&\AKh(L_1)\{-1\}\arrow[ul]
\end{tikzcd}
\end{equation}

Here $n_-$ is the number of negative crossings in the diagram for $L$, $n_-^0$ is the number of negative crossings in the diagram for $L_0$, $\{a\}$ is a shift in the quantum grading by $a$ and $[b]$ is a shift in the homological grading by $b$. Corresponding to resolving a positive crossing we have the following exact triangle:

\begin{equation}\label{eq:akhskeinpos}
\begin{tikzcd}
\AKh(L)\arrow[rr]&&\AKh(L_0)\{1\}\arrow[dl]\\&\AKh(L_1)[n_-^1-n_-+1]\{3n_-^1-3n_-+2\}\arrow[ul]
\end{tikzcd}\end{equation}

Grigsby-Licata-Wehrli showed that for annular Khovanov homology with complex coefficients carries the structure of an $\slrep(\C)$ representation~\cite{grigsby_annular_2018}. This entails that{, up to overall grading shift in the quantum grading,} $\AKh(L;\C)$ decomposes as a direct sum of vector spaces $V_n^{i}$, where $V_n^{i}$ is the rank $n+1$ vector space supported in homological grading $i$ and quantum and annular gradings $(-n+2p,-n+2p)$ for all $0\leq p\leq n$. {It is not hard to see that annular Khovanov homology of an annular link $L$ is supported in annular gradings $j$ with $|j|\leq w(L)$, where $w(L)$ is the \emph{wrapping number} of $L$; i.e. the minimum \emph{geometric} intersection number of $L$ with a meridional disk for the thickened annulus. Consequently, $\AKh(L;\C)$ can only contain $V_n$ summands with $n\leq w(L)$.}

Annular Khovanov homology admits a spectral sequence to the Khovanov homology of the underlying link. The differential on annular Khovanov homology inducing this spectral sequence increases the homological grading by one, preserves the quantum grading and decreases the annular grading. Moreover, the differential forms part of an action of the $\slrep(\wedge)$ current algebra on $\AKh(L;\C)$ --- a stronger structural property than being an $\slrep(\C)$ representation. See~\cite[Section 6]{grigsby_annular_2018} for details.

 {It is not hard to see that ${\AKh(b(\beta),k=-n;R)}$ consists of a single copy of $R$. This summand is generated by Plamenevskaya's \emph{transverse invariant}}~\cite{Plamenevskaya06transverse}{. This class consists of $n$ concentric circles about the braid axis, each decorated with an $X$. The quantum grading of this generator is the self linking number of $\beta$ --- which we denote $\slflk(\beta)$ --- and the homological grading is zero.}
\subsection{From orderability to rank bounds}\label{subsec:rankbound}
In this section we prove Theorem~\ref{thm:rankboundforarbbraid}, and various related results. The most concise version of our result is that stated in the introduction:

\begin{theorem}\label{thm:rankboundforarbbraid}
If $\beta$ be an $n$-braid with $n\geq 2$ then:\begin{enumerate}
    \item $\rank(\AKh(b(\beta);\C))\geq 2n$.
\item  $\rank(\AKh(c(\beta);\C))\geq 4n$.

\end{enumerate}
\end{theorem}

Of course, if $n=1$, then  $c(\beta)$ is undefined, while $\rank(\AKh(b(\beta);\C))=2$. Theorem~\ref{thm:rankboundforarbbraid} is a direct consequence of the next --- stronger --- result. To state it recall that the braid group is left orderable. There are many different interpretations of the ordering on the braid group {(see}~\cite{dehornoy2008ordering}), two of which we will use in this section. The first is the following; we write $\beta<1$ if there is a word for $\beta$ in the letters given by the standard Artin generators and their inverses which is \emph{$\sigma$-negative}; i.e. if among the letters that occur in that word, the letter of the lowest index occurs only with negative powers. {See}~\cite[Chapter II, Section 1.2]{dehornoy2008ordering} for further details.

\begin{lemma}\label{lem:nontivialrep}

Suppose $\beta$ is a $\sigma$-negative $n$-braid. Then $\AKh(c(\beta);\C)$ contains a 
\begin{align*}V_{n}^{a_1}\{a_2\}\oplus V_{n}^{b_1}\{b_2\}\oplus V_{n-2}^{a_1-1}\{a_2-2\}\oplus V_{n-2}^{b_1-1}\{b_2-2\}\end{align*}

\noindent summand. Moreover{,} the $\slrep(\wedge)$ action sends the lowest annular grading generator in ${V_{n-2}^{a_1-1}\{a_2-2\}}$ to the lowest annular grading generator in $V_{n}^{a_1}\{a_2\}$ and the lowest annular grading generator in $V_{n-2}^{b_1-1}\{b_2-2\}$ to the lowest annular grading generator in $V_{n}^{b_1}\{b_2\}$.

Similarly, $\AKh(b(\beta);\C)$ contains a $V_{n}^0\{n+\slflk(\beta)\}\oplus V_{n-2}^{-1}\{n-2+\slflk(\beta)\}$ summand. Moreover{,} the $\slrep(\wedge)$ action sends the lowest annular grading generator in $V_{n-2}^{-1}\{n+\slflk(\beta)-2\}$ to the lowest annular grading generator in $V_{n}^0\{n+\slflk(\beta)\}$.
\end{lemma} 

\begin{remark}As we shall see, one can write down the values of the quantum and homological gradings of $\AKh(c(\beta),k=-n;R)$ in terms of diagrammatic data for $\beta$. {Comparing this to the braid closure case --- where there is a single generator whose homological grading is zero and quantum grading is the self linking number --- } it is natural to ask what topological{ or contact geometry-theoretic} information these numbers contain. We do not pursue this question here.
\end{remark}
Our strategy for the proof of Lemma~\ref{lem:nontivialrep} is to use properties of $\sigma$-negative words to control the annular Khovanov homology of closures of braids in the next to minimal annular grading.

Given an annular link $L$ view $\CKh(L;R)$ as a chain complex filtered with respect to the annular filtration. The differential comes in two pieces, $\partial_0+\partial_{-2}$, where $\partial_0$ preserves the annular grading on {$\CKh(L;R)$} and $\partial_{-2}$ decreases it b{y} $2$. $\AKh(L;R)$ can be viewed as $(\CKh(L;R),\partial_0)$ i.e. the first page of the corresponding spectral sequence.

\begin{lemma}\label{lem:chainlevel}
Let $\beta$ be a $\sigma$-negative braid which is not of index $1$. Then there are chain maps \begin{align*} 
f_c :\CKh(c(\beta);i,j,k\leq -n;R)\to\CAKh(c(\beta),i-1,j,2-n;R)\end{align*}
and 
\begin{align*}{f_b:\CKh(b(\beta);i,j,k\leq -n;R)\to\CAKh(b(\beta);i-1,j,2-n;R)}.\end{align*}

Moreover, $\partial_{-2}^*$ is a left inverse to $f_c$ or $f_b$ on the $E_1$ page of the spectral sequence from $\CAKh(c(\beta);R)$ or  $\CAKh(b(\beta);R)$ to $\Kh(b(\beta);R)$. 
\end{lemma}

{Here by $\CKh(c(\beta);i,j,k\leq -n;R)$, we mean the $k$-filtered part of filtration level $n$.} It follows that $\AKh(c(\beta),k=2-n;R)$ has a{n} $\AKh(c(\beta),k=-n;R)[-1]$ summand while ${\AKh(b(\beta), k=2-n;R)}$ has a $\AKh(b(\beta),k=-n;R)[-1]$ summand. Here{,} $[-1]$ indicates a shift in the homological grading by $-1$.

\begin{proof}

We treat the case of clasp-closures. The proof in the braid-closure case is the same in essence and strictly easier in practice.

Since $\beta$ is $\sigma$-negative $\beta$ is isotopic to a braid $\beta'$ that contains the inverse of an Artin generator $\sigma_i^{-1}$ but not the corresponding Artin generator, $\sigma_i$. Consider the diagram $D$ for $c(\beta')$ as in Figure~\ref{fig:claspclosure}. There are three complete resolutions $D_1,D_2$ and $D_3$ corresponding to ${\CAKh(c(\beta),k=-n;R)}$; these are shown in Figure~\ref{fig:CAKh-n}. There are four generators of $\CAKh(c(\beta),k=-n;R)$. They can be described as follows; for each $i$ we have a generator $\mathbf{X}_i$ where every circle in the resolution $D_i$ is decorated with an $X$. We have a final generator $\mathbf{1}$ which corresponds to decorating the homologically essential circles in diagram $D_1$ with $X$s and the homologically inessential circle with a $1$. The non-trivial components of the differential are given by $\partial_{0}\mathbf{X}_2=\partial_0\mathbf{X}_3=\mathbf{1}$ for an appropriate sign assignment.

\begin{center}
 \begin{figure}[h]
 \includegraphics[width=5cm]{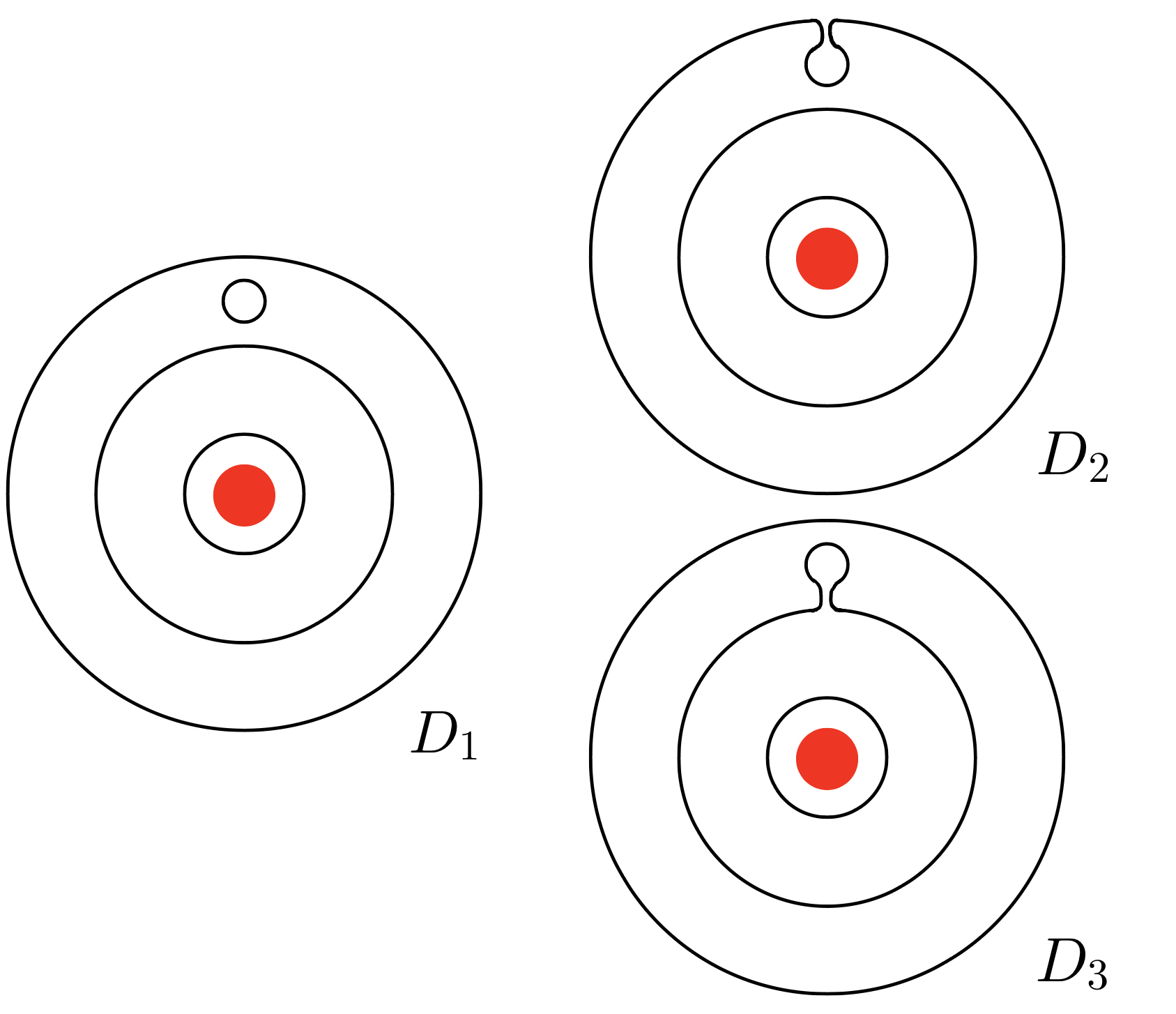}
    \caption{The resolutions for the canonical diagram --- as in Figure~\ref{fig:claspclosure} --- for a clasp-closure yielding generators of $\CAKh(c(\beta),k=-n;R)$. The solid red dot is the annular axis.}\label{fig:CAKh-n}
   \end{figure}
\end{center}

Pick one of the crossings corresponding in $D$ to a letter $\sigma_i^{-1}$ and label it {$y$}. Consider the resolutions $D_i'$ of $D$ that are identical to the resolutions $D_i$ aside from at {$y$}. Given a generator $\mathbf{x}\in\CAKh(c(\beta),R;k=-n)$ define $f_c(\mathbf{x})\in\CAKh(c(\beta),R;k=2-n)$ to be the generator which agrees with $\mathbf{X}_i$ on every circle in the resolution that does not involve {$y$}, and is labeled with an $X$ on the remaining circle. Observe that if {$\mathbf{x}$} is of $(i,j,k)$ grading $(a,b,-n)$ then $f_c(\mathbf{x})$ is of grading $(a-1,b,2-n)$.

Since $\beta$ is $\sigma$-negative, $f_c$ is a chain map viewed as a map $(\CKh(c(\beta);R),\partial)\to(\CKh(c(\beta);R),\partial_0)$. To verify this notice that the maps corresponding to changing the resolutions in the $\beta'$ part of the diagram corresponds to merging circles decorated with $X$'s. Thus the only contributions to the differential on $\CAKh(c(\beta),k=2-n;R)$ involve the crossings contained in the part of the diagram for the clasp.

One can check that $\partial_{-2}^*$ is a left-inverse to $f_c$ similarly; the only contributions to the differential which lower the annular filtration level correspond to changing the resolution at $y$. This corresponds to splitting a circle labeled with an $X$, resulting in two circles both labeled with $X$'s.
\end{proof}

\begin{proof}[Proof of Lemma~\ref{lem:nontivialrep}]
We treat only the braid-closure case, since the clasp-closure case is essentially the same. Observe first that by the previous lemma, $f_b$ induces an injection \begin{align*}
f_b^*:\AKh(b(\beta),k=-n;\C)\cong \C[0,sl(\beta),-n]\inj \AKh(b(\beta),(-1,sl(\beta),2-n);\C).\end{align*}

Here $\C[0,sl(\beta),-n]$ indicates a $\C$ summand supported in {$(i,j,k)$} grading $(0,sl(\beta),-n)$. Now, $\AKh(b(\beta);\C)$ carries the structure of an $\slrep(\C)$-representation, where each summand is supported in a single homological grading. {Observe that the summand $\C[0,sl(\beta),-n]$ and its image under $f_b^*$ are supported in different homological gradings and consequently cannot be part of the same irreducible $\slrep(\C)$ representation.} $f_b^*(\C[0,sl(\beta),-n])$ {must be a minimal annular grading $\C$-summand of its $\slrep(\C)$ representation, as else we would have that $\rank(\AKh(b(\beta),k=-n;\C))\geq 2$, which is a contradiction since $b(\beta)$ is a braid-closure.} It follows that the $\AKh(b(\beta);\C)$ contains the two desired representations as summands.

The structure of these summands as an $\slrep(\wedge)$ representation follow from the fact that  $\partial_{-2}^*$ is part of the $\slrep(\wedge)$ action.
\end{proof}

We can now extract a rank bound for annular Khovanov homology with $\Z/2$ coefficients from Lemma~\ref{lem:nontivialrep} and the proof of Lemma~\ref{lem:chainlevel}.

\begin{lemma}\label{lem:akhboundmod2}
    Let $\beta$ be a non-identity $n>1$-braid. Then; \begin{align*}  
    \rank(\Kh(b(\beta));\Z/2)\leq \rank(\AKh(b(\beta);\Z/2))-2(n-1),\end{align*}

\noindent    while;
    \begin{align*}     
    \rank(\Kh(c(\beta);\Z/2)))\leq \rank(\AKh(c(\beta);\Z/2))-4(n-1).\end{align*}
\end{lemma}

\begin{proof}

Suppose $\beta$ is as in the statement of the Lemma. Since any non-identity braid $\beta$ is either $\sigma$-negative or $\sigma$-positive, we have four cases to consider. We prove the result in the $\sigma$-negative braid-closure case. The other three cases are similar.

It suffices to show that the rank of the map $\partial_{-2}^*:\AKh(b(\beta);\Z/2)\to\AKh(b(\beta);\Z/2)$ is at least $n-1$. The universal coefficient theorem for homology is functorial, so considering the map $\partial_{-2}^*$ we obtain the following commutative diagram:

\begin{center}
\begin{equation}
\begin{tikzcd}
    0\arrow[r]&\AKh_{i}(b(\beta);\Z)\otimes\F \arrow[r,"\mu"]\arrow[d,"\partial_{-2}^*\otimes \mathbf{1}"]&\AKh_i(b(\beta);\F)\arrow[r]\arrow[d,"\partial_{-2}^*"]&\Tor^1_\Z(\AKh_{i+1}(b(\beta);\Z),\F)\arrow[r,"h"]\arrow[d]&0\\ 0\arrow[r]&\AKh_{i+1}(b(\beta);\Z)\otimes\F \arrow[r,"\mu"]&\AKh_{i+1}(b(\beta);\F)\arrow[r]&\Tor^1_\Z(\AKh_{i+2}(b(\beta);\Z),\F)\arrow[r,"h"]&0
\end{tikzcd}\label{eq:naturaluc}
\end{equation}\end{center}

Here $\mu$ is the map defined on elementary tensors given by $\mu(\mathbf{x}\otimes a)\mapsto a\mathbf{x}$. Taking $\F=\C$, so that $\Tor^1_\Z(\AKh_{j}(b(\beta);\Z),\F)$ vanishes for all $j$,  and setting $i=-1$ we obtain:

\begin{center}
\begin{equation}
\begin{tikzcd}
    0\arrow[r]&\AKh_{-1}(b(\beta);\Z)\otimes\C \arrow[r,"\mu"]\arrow[d,"\partial_{-2}^*\otimes \mathbf{1}"]&\AKh_{-1}(b(\beta);\C)\arrow[r]\arrow[d,"\partial_{-2}^*"]&0\\ 0\arrow[r]&\AKh_{0}(b(\beta);\Z)\otimes\C \arrow[r,"\mu"]&\AKh_0(b(\beta);\C)\arrow[r]&0
\end{tikzcd}\label{eq:universC}
\end{equation}\end{center}

Now, from the proof of Lemma~\ref{lem:chainlevel}, {${\partial_{-2}^*:\AKh_{-1}(b(\beta);\C)\to \AKh_0(b(\beta);\C)}$} has a component given by the identity map {$\C^{\oplus(n-1)}\to\C^{\oplus(n-1)}$} with respect to the canonical basis for $\AKh(b(\beta);\C)$. {More specifically, observe that since $\partial_{-2}^*$ is non-trivial on a bottom generator of $V_{n-2}^{-1}\{n-2+sl(\beta)\}$ and $\partial_{-2}^*$ is part of the structure of $\AKh(b(\beta);\C)$ as an $\slrep(\wedge)$-representation, we have that $\partial_2^*$ is non-trivial on the entire $V_{n-2}^{-1}\{n-2+sl(\beta)\}$ summand. It follows from Diagram}~(\ref{eq:universC}) {that $\AKh_{-1}(b(\beta);\Z)$ and $\AKh_{0}(b(\beta);\Z)$ contain $\Z^{\oplus (n-1)}$ summands and that $\partial_{-2}^*$ acts as the identity map between these two summands.} 

Now take $\F=\Z/2$ in {D}iagram~(\ref{eq:naturaluc}). By the commutativity of the diagram we deduce that {$\partial_{-2}^*:\AKh_{0}(b(\beta);\Z/2)\to\AKh_{0}(b(\beta);\Z/2)$} has a component given by the identity map {$(\Z/2)^{\oplus(n-1)}\to(\Z/2)^{\oplus(n-1)}$}.{The result follows.}
\end{proof}

\begin{proof}[Proof of Theorem~\ref{thm:rankboundforarbbraid}]
    Suppose $\beta$ is a braid. Then $\beta$ is either $\sigma$-positive, $\sigma$-negative or the identity. If $\beta$ is the identity then $\rank(\AKh(b(\beta);\C))=2^{n}\geq 2n$.

    If $\beta$ is $\sigma$-negative the result follows immediately from {Lemma}~\ref{lem:nontivialrep}. If $\beta$ is $\sigma$-positive, the result follows from applying {Lemma}~\ref{lem:nontivialrep} to the mirror of $\beta$ --- which is $\sigma$-negative --- and applying symmetry properties of annular Khovanov homology.
\end{proof}

We can also prove an annular Khovanov homology analogue of a result of Ni from knot Floer homology~\cite[Theorem A.1]{ni2020exceptional}. To do so{,} we exploit a geometric interpretation of the ordering of the braid group in terms of curve diagrams; see~\cite[Chapter 10]{dehornoy2008ordering}.  Recall that $n$-braids can be viewed as mapping classes of $n$-punctured disks. Recall too that a braid is \emph{right (left) veering} if it sends every \emph{admissible} arc to the right (left). See~\cite[Section 3.1]{BaldwinGrigsby} for a definition of admissible. If a braid is non right (left) veering then it is conjugate to a $\sigma$-negative (positive) braid --- see the proof of~\cite[Proposition 3.1]{BaldwinGrigsby}.

\begin{proposition}
    Suppose $\beta$ is a non right-veering and non left-veering $n$-braid, with $n\geq 4$. Then {$\rank(\AKh(b(\beta);\C))\geq 4n-4$}.%=\neq n+1+2(n), =n+1+n-1+n-1+n-4+1. 
\end{proposition}

\begin{proof}
    Suppose $\beta$ is as in the statement of the proposition. Since $V$ is an $n$-braid there is a $V_n^0$ summand. Since $\beta$ is non right-veering, it is conjugate to a braid $\beta'$ that is $\sigma$-negative. Since $b(\beta')=b(\beta)$, $\AKh(b(\beta);\C)$ has a $V_{n-2}^{-1}$-summand by Lemma~\ref{lem:nontivialrep}. Similarly, since $\beta$ is non left{-}veering{,} $\AKh(b(\beta);\C)$ contains a $V_{n-2}^1$ summand  by Lemma~\ref{lem:nontivialrep}. 
    
    Assume towards a contradiction that there is a unique generator in $(i,k)$-grading $(0,4-n)$. Consider the generator $\mathbf{x}$ of $(i,k)$- grading {${(-1,6-n)}$}. Consider ${\partial_{-2}^*:\AKh(b(\beta);\C)\to\AKh(b(\beta);\C)}$. Since $b(\beta)$ is is isotopic to the braid-closures of $\sigma$-positive and $\sigma$-negative words, Lemma~\ref{lem:nontivialrep} {and the symmetry properties of the spectral sequence from annular Khovanov homology to Khovanov homology under mirroring }impl{y} that $(\partial_{-2}^*)^2(\mathbf{x})\neq 0$, a contradiction. The result now follows from noting that $\AKh(b(\beta);\C)$ carries the structure of an $\slrep(\C)$-representation; {specifically there must be at least one more generator in $(i,k)$-grading $(0,4-n)$, and so in turn another $V_{n-4}$-summand}.
\end{proof}

In the case that $n=3$ we have the following;

\begin{proposition}
    Suppose $\beta$ is a $3$-braid that i{s} non right-veering and non left-veering. Then $\rank(\AKh(b(\beta);\C)\geq10 $%4+2+2
\end{proposition}

\begin{proof}
     Since $V$ is a $3$-braid there is a $V_3^0$ summand. Since $\beta$ is non right-veering there is a $V_{1}^{-1}$ summand by Lemma~\ref{lem:nontivialrep}. Since $\beta$ is non left{-}veering there is a $V_{1}^1$ summand by Lemma~\ref{lem:nontivialrep}. Assume towards a contradiction that {$\rank(\AKh(b(\beta);\C))<8$}. Then in fact:\begin{align*}
        \AKh(b(\beta);\C)\cong V_1^{-1}\oplus V_3^{0}\oplus V_1^{1}.
      \end{align*} {Lemma}~\ref{lem:nontivialrep} implies that all generators die under the spectral sequence from  $\AKh(b(\beta);\C)$ to $\Kh(b(\beta);\C)$, a contradiction. Thus $\rank(\AKh(b(\beta);\C))>8$. The result now follows from the fact that  $\rank(\AKh(b(\beta);\C))$ is even for $3$-braids --- which have wrapping number $3$ --- since it splits as a direct sum of $V_3$ and $V_1$ summands.
\end{proof}

It is unclear to the author if similar results could be obtained for clasp-closures, since clasp-closures of conjugate braids are not necessarily isotopic, so the proof strategy above break{s} down.

\subsection{Applications of the Rank Bound}\label{sec:applications}

Let $\beta_n$ denote the $n$-braid $\sigma_1\sigma_2\dots\sigma_{n-1}$, and $\mathbbm{1}_n$ denote the identity $n$-braid.

\begin{proposition}
    
\label{prop:ranktohomology}
    Suppose $\alpha$ is an $n$-braid, with $n>2$. If $\rank(\AKh(b(\alpha);\C))=2n$ then $\AKh(b(\alpha);\C)\cong\AKh(b(\beta_n^{\pm 1});\C)$.
\end{proposition}

We note the $n=1$ case is uninteresting, as is the $n=2$ case since the annular Khovanov homology of all $2$-braids is known~\cite{grigsby_annular_2018}. Indeed, in the $2$-braid case, the proposition is false; $\rank(\AKh(b(\mathbbm{1}_2);\Z)=\rank(\AKh(b(\beta_2^{\pm 1});\Z))=4$~\cite{grigsby_annular_2018}.

\begin{proof}
    Suppose $\alpha$ is neither $\sigma$-positive nor $\sigma$-negative. Then $\alpha$ is the identity braid, and one can readily check that $\rank(\AKh(b(\alpha);\C))=2^n$. It follows that $n=1$ or $n=2$, contradicting our assumption.

    Suppose now that $\alpha$ is $\sigma$-negative. Then $\AKh(b(\alpha);\C)$ contains a ${V_{n}^0\{n+\slflk(\beta)\}\oplus V_{n-2}^{-1}\{\slflk(\beta)-2\}}$ summand by Lemma~\ref{lem:nontivialrep}, so must in fact be $V_{n}^0\{\slflk(\beta)\}\oplus V_{n-2}^{-1}\{\slflk(\beta)-2\}$
    The $\sigma$-positive case follows by a similar argument.
\end{proof}

Since annular Khovanov homology detects $\beta_1,\beta_2,\beta_3,\beta_4,\beta_5,\beta_6,\beta_8,\beta_{10}$~\cite{farber2022fixed},\cite{binns2020knot},\cite{binns2022rank} it follows from Proposition~\ref{prop:ranktohomology} that $\rank(\AKh(-;\Z/2))$ detects each of these braids amongst braids of the correct index. That is, we have the following corollary;

\begin{corollary}\label{rankclassification}
Suppose $\alpha$ is an $n$-braid with $n\in\{ 3,4,5,6,8,10\}$. If ${\rank(\AKh(b(\alpha);\Z/2))=2n}$ then $\alpha=\beta_n^{\pm 1}$.
\end{corollary}

This allows one to cut some casework from Baldwin-Hu-Sivek's proof that Khovanov homology detects $T(2,5)$ with $\Z/2$ coefficients~\cite[Sections 4 \& 5]{baldwin2021khovanov}. In fact{,} we can generalise part of Baldwin-Hu-Sivek's argument as follows:

\begin{proposition}
Suppose $K$ is an $m$-periodic link with axis of symmetry $A$, ${\rank(\Kh(K;\Z/m))\leq2n}$ and $\lk(A,K)\geq n$. Let $J$ denote the quotient of $K$ viewed as an annular link about $A$. Then ${\AKh(J;\C)\cong \AKh(b(\beta_n^{\pm 1});\C)}$. Moreover, if $n\in\{3,4,5,6,8,10\}$, then $J=b(\beta_n^{\pm 1})$.
\end{proposition}

We follow the argument of Baldwin-Hu-Sivek more or less verbatim.
\begin{proof}
Let $J$ denote the quotient of $K$, viewed as an {a}nnular link about $A$. A result of Stoffregen-Zhang~\cite{stoffregen2018localization} implies that \begin{align*}
\rank(\AKh(J;\Z/m))\leq \rank(\Kh(K;\Z/m)).\end{align*} Thus {$\rank(\AKh(J;\C))\leq\rank(\AKh(J;\Z/m))\leq2n$, with the first inequality coming from} the universal {coefficient} theorem. Xie defined a spectral sequence from annular Khovanov homology to an invariant called \emph{annular instanton Floer homology} which respects the annular grading~\cite{XieInstantonsandannularKhovanovhomology}. Xie-Zhang showed {in}~\cite[Theorem 1.6]{xie_instanton_2019} that the maximum non-trivial {a}nnular grading of {the} annular instanton Floer homology {of $J$} is given by $$\min\{2g(S)+|S\cap J|:S\text{ is a meridional surface}\}.$$ {Here, a meridional surface in the thickened annulus is any surface which meets the boundary of the thickened annulus in a meridian, i.e. a curve which bounds a disk.} Note that $\min\{2g(S)+|S\cap J|\}\geq\lk(A,J)\geq n$. It follows that $\rank(\AKh(J;\C,j=n))>0$. Therefore{,} $J$ is a braid by~\cite[Theorem 1.1]{grigsby_sutured_2014}, since $\AKh(J;\C)$ contains a copy of $V_{n+1}$ and so cannot have rank greater than one in the maximum annular grading. The result then follows directly from Proposition~\ref{prop:ranktohomology} in general and Corollary~\ref{rankclassification} in the special cases.
\end{proof}

\subsection{Braid-closures}\label{sec:akhbraids}

In this section we prove the following result:

\begin{theorem}\label{thm:AKH}
    Annular Khovanov homology with integer coefficients detects $b(\sigma_1\sigma_2^{n})$ for $-2\leq n\leq 5$.
\end{theorem}

We remind the reader that the $n=1$ case was already proven in~\cite{binns2020knot}, so we do not discuss it here. We begin with some computations.

We compute the annular Khovanov homology of the annular links $b(\sigma_1\sigma_2^n)$. It is readily checked that \begin{equation}\label{eq:n=0case}
\AKh(b(\sigma_1);\C)\cong V_3^0\{1\}\oplus V_1^0\{1\}\oplus V_1^1\{3\}\end{equation} and that $\AKh(b(\sigma_1);\Z/2)$ can be obtained by replacing each homogeneous $\C$ summand in $\AKh(c(\sigma_1);\C)$ with a $\Z/2$ summand. We compute the annular Khovanov homology of the remaining links.
\begin{lemma}\label{lem:t2ncomps}
    For $n\geq 1${,} $\AKh(b(\sigma_1\sigma_2^n);\C)$ is given by;

    \begin{align*}
        V^0_3\{n+1\}\oplus V_1^1\{n+3\}\oplus\underset{1\leq i\leq n-1}{\bigoplus} V_1^{1+i}\{n+1+2i\}.
    \end{align*}

    For $n\leq -1${,} $\AKh(b(\sigma_1\sigma_2^n);\C)$ is given by;
  \begin{align*}
        V^0_3\{1+n\}\oplus V_1^1\{3+n\}\oplus V_1^0\{1+n\}\oplus \underset{-1\geq i\geq n}{\bigoplus} V_1^{i}\{n+1+2i\}.
    \end{align*}

    In each case $\AKh(b(\sigma_1\sigma_2^n);\Z/2)$ is given by replacing each homogeneous $\C$-summand with a $\Z/2$-summand.
\end{lemma}

\begin{proof}
Consider the standard diagram for $b(\sigma_1\sigma_2^n)$, as in Figure~\ref{fig:braidclosure}. We consider two cases; that in which $n\geq 1$ and that in which $n\leq -1$. We proceed by induction in both instances.

In the $n\geq 1$ case, note that {$\AKh(b(\sigma_1\sigma_2);\C)\cong V_3^0\{2\}\oplus V_1^{1}\{4\}$.}

For the inductive step we resolve $b(\sigma_1\sigma_2^n)$ at one of the crossings corresponding to a $\sigma_2$. Observe that the $1${-}resolution is the braid-closure of the identity $1$-braid, while the $0${-}resolution is $b(\sigma_1\sigma_2^{n-1})$. Applying the exact triangle~(\ref{eq:akhskeinpos}) and the fact that $n_-=0$, and $n_-^1=n-1$ we obtain:

\begin{center}
\begin{tikzcd}
\AKh(b(\sigma_1\sigma_2^{n});\C)\arrow[rr]&&\AKh(b(\sigma_1\sigma_2^{n-1});\C)\{1\}\arrow[dl,"\delta"]\\&\AKh(b(\mathbf{1}_1);\C)[n]\{3n-1\}\arrow[ul]

\end{tikzcd}
\end{center}

Now, \begin{equation}\label{eq:identity1}
\AKh(\mathbf{1}_1;\C)\cong V_1^0.
\end{equation} For $n\geq 2$ this maps splits by inductive hypothesis{, since there are no generators in $\AKh(b(\sigma_1\sigma_2^{n-1});\C)$ of the correct gradings to map non-trivially to $V_1^0[n]\{3n-1\}$}. For $n=2$ the result can be computed by hand or one can note that the connecting map $\delta$, which increases the $i$ grading by $1$, must vanish as $\Kh(b(\sigma_1\sigma_2^2);\C)\cong\Kh(T(2,2));\C)$ has two generators with $i$ grading $2$.

We now proceed to the {$n\leq -1$} case. Note that:\begin{align*}\AKh(b(\sigma_1\sigma_2^{-1});\C)\cong V_3^0\oplus V_1^{1}\{2\}\oplus V_1^0\oplus V_1^{-1}\{-2\}.\end{align*}

For the inductive step we resolve $b(\sigma_1\sigma_2^n)$ at one of the crossings corresponding to a $\sigma_2^{-1}$. Applying the exact triangle~(\ref{eq:akhskeinneg}) and noting that {$n_-=-n$}, and {$n_-^0=0$} we obtain:

\begin{center}
\begin{equation*}
\begin{tikzcd}
    \AKh(b(\sigma_1\sigma_2^n);\C)\arrow[rr]&&\AKh(b(\mathbf{1}_1);\C)[n]\{1+3n\}\arrow[dl]\\&\AKh(b(\sigma_1\sigma_2^{n-1});\C)\{-1\}\arrow[ul]
\end{tikzcd}
\end{equation*}\end{center}
 Given Equation~(\ref{eq:identity1}) and the inductive hypothesis, {the grading data implies that} the exact triangle{ must} split  {. The result follows.}
 
 Finally, to see that $\AKh(b(\sigma_1\sigma_2^n);\Z/2))$ is as claimed, observe that the proofs above from the case of complex coefficients carry through to the case of $\Z/2$ coefficients verbatim.\end{proof}

Note that for $\F\in\{\Z/2,\C\}$, $\AKh(b(\sigma_1^{-1}\sigma_2^n);\F)$ can be determined from Lemma~\ref{lem:t2ncomps} using symmetry properties of annular Khovanov homology. In particular, given that the $3$-braid representatives of the link $T(2,n)$ with $n\neq 0$ are exactly links of the form $b(\sigma_1^{-1}\sigma_2^n)$ and $b(\sigma_1\sigma_2^n)$ by Birman-Menasco's classification of $3$-braids~\cite{BirmanMenasco3braids}, we have computed the annular Khovanov homology of all $3$-braid representatives of $T(2,n)$.

We now proceed to prove Theorem~\ref{thm:AKH}. Our strategy is to use the spectral sequence from the annular Khovanov homology of an annular link to Khovanov homology of the underlying link to determine the underlying link type then to exploit Birman-Menasco's classification of $3$-braids~\cite{BirmanMenasco3braids}.

\begin{proof}[Proof of Theorem~\ref{thm:AKH}]

Suppose $L$ is an annular link with $\AKh(L;\Z)\cong\AKh(b(\sigma_1\sigma_2^{n});\Z)$ for some $n$. Note that $\AKh(L;R)\cong\AKh(b(\sigma_1\sigma_2^{n});R)$ for $R\in\{\Q,\C,\Z/2\}$ by the universal coefficient theorem. Since $L$ has rank one in the maximum non-trivial $k$ grading it follows that $L$ is isotopic to the closure of a braid $\beta$~\cite{grigsby_sutured_2014}. Since the maximum non-trivial $k$ grading is $3$ it follows that $\beta$ has index $3$. We now split our analysis into three cases; $n=-1$, $n=-2$ and $n\geq 0$.

${\mathbf{n=-1}}$.
\noindent We claim that $b(\beta)$ is an unknot. First note that since $\AKh(b(\beta);\C)$ has support in odd quantum gradings it follows that $b(\beta)$ has an odd number of components. Consider the spectral sequence from $\AKh(b(\beta);\Z/2)$ to $\Kh(b(\beta);\Z/2)$. Suppose $L$ is neither $\sigma$-positive nor $\sigma$-negative. Then $L$ is the identity $3$-braid. This is a contradiction{,} since the identity $3$-braid has annular Khovanov homology of rank $8$. It follows that $L$ is either $\sigma$-positive or $\sigma$-negative. Thus $\rank(\Kh(b(\beta);\Z/2))\leq 6$ by Lemma~\ref{lem:akhboundmod2}. It follows that $L$ has at most two components, since $\rank(\Kh(b(\beta);\Z/2))\geq 2^m$ where $m$ is the number of components of $L$. Thus{,} $\rank(\Khr(b(\beta);\Z/2))\leq 3$ by~\cite[Corollary 3.2.C]{shumakovitch2014torsion}. Thus we have that $L$ is either a trefoil or an unknot~\cite{baldwin2018khovanov,kronheimer2011khovanov}. $L$ cannot be a trefoil, since $\Kh(T(2,\pm 3);\Q)$ has support in quantum gradings $\pm 9$. It follows that $L$ is an unknot. Since there are only three $3$-braids representing the unknot up to conjugation by Murasugi's classification~\cite{Murasugiclosed3}, it suffices to show that $\beta$ is not  $\sigma_1\sigma_2$ or $\sigma_1^{-1}\sigma_2^{-1}$. But these two braids have braid-closures with annular Khovanov homology of rank $6$ over the complex numbers, rather than $10$, completing the proof in this case.

${\mathbf{n=-2}}$. First{,} note that since $\AKh(b(\beta);\C)$ has support in even quantum gradings $b(\beta)$ has an even number of components. Since $\beta$ is a $3$-braid it follows that $\beta$ has two components. Observe that $\rank(\Kh(b(\beta);\Z/2))\leq 8$ by Lemma~\ref{lem:akhboundmod2}, so by~\cite[Corollary1.4]{xie2020links} $b(\beta)$ represents
 a two{-}component unlink, $T(2,\pm 2)$ or $T(2,\pm 4)$. By Birman-Menasco's classification result~\cite{BirmanMenasco3braids}, $c(\beta)$ must be of the form $\sigma_1^{\pm1}\sigma_2^n$ for some even $|n|\leq 4$. Annular Khovanov homology distinguishes each of these links, concluding the proof in this case.

$\mathbf{n\geq 0}$.
\noindent Observe that $\beta$ cannot be the identity braid, since its annular Khovanov homology is not of the correct form. Moreover, $\beta$ cannot be $\sigma$-negative as there are no generators of $\AKh(L;\C)$ in homological grading $-1$. It follows that $\beta$ is $\sigma$-positive. An application of Lemma~\ref{lem:akhboundmod2} {and Lemma}~\ref{lem:t2ncomps} implies that {for $n\geq 1$} $$\rank(\Kh(b(\beta);\Z/2))\leq \rank(\Kh(b(\beta);\Z/2))\leq 2n.$$ {On the other hand, for $n=0$, Equation}~(\ref{eq:n=0case}) {and Lemma}~\ref{lem:akhboundmod2} {imply that }$$\rank(\Kh(b(\beta);\Z/2))\leq \rank(\Kh(b(\beta);\Z/2))\leq 4.$$ Thus {$\rank(\Khr(b(\beta);\Z/2))\leq \min\{n,2\}$} by~\cite[Corollary 3.2.C]{shumakovitch2014torsion}. We now treat our three subcases:

$\mathbf{n=3}$.
\noindent In this case{,} $\rank(\Khr(b(\beta);\Z/2))\leq 3$. Since $\AKh(b(\beta);\Q)$ is supported in odd quantum gradings it has an odd number of components. Note that $b(\beta)$ can have no more than two components, since $\rank(\Kh(b(\beta);\Q))\geq 2^m$ where $m$ is the number of components of $c(\beta)$.  It follows that $b(\beta)$ is a knot. Now, $\rank(\Khr(b(\beta);\Z/2)))$ is odd, so that $\rank(\Khr(b(\beta);\Z/2)))=1$ or $\rank(\Khr(b(\beta);\Z/2)))=3$. If $\rank(\Khr(b(\beta);\Z/2)))=1$ then $b(\beta)$ represents the unknot by~\cite{kronheimer2011khovanov}. But the three braid-closures of $3$-braids representing the unknot have different annular Khovanov homology from $\AKh(b(\beta);\Z)$, so $n\neq 1$. It follows that $n=3$ and $b(\beta)$ represents a trefoil by~\cite{baldwin2018khovanov}. There are four $3$-braids representing trefoils by~\cite{BirmanMenasco3braids}. They each have distinct annular Khovanov homology by Lemma~\ref{lem:t2ncomps}, so the result follows.

$\mathbf{n=5}$.
Since $\AKh(b(\beta);\C)$ is supported in odd quantum gradings{,} $b(\beta)$ has an odd number of components. Since $\beta$ is a $3$-braid, $b(\beta)$ has either one or three components. If $\beta$ has three components, then{ $\rank(\Kh(b(\beta);\C))\leq\rank(\AKh(b(\beta);\C))=14$, so} Batson-Seed's link splitting spectral sequence implies that each component of $b(\beta)$ is unknotted. {More specifically, observe that for any component $K$ of $b(\beta)$, we have that} $${2\rank(\Khr(K;\Z/2))\leq \rank(\Kh(K;\Z/2))\leq 3}$$ {by the universal coefficient theorem and}~\cite[Corollary 3.2.C]{shumakovitch2014torsion}{. Consequently,} $$\rank(\Khr(K;\Z/2))=1$$ {and $K$ is the unknot by}~\cite{kronheimer2011khovanov}. Birman-Menasco's classification theorem~\cite{BirmanMenasco3braids} implies that the only $3$-braid representative of the {three-component unlink} is the identity {$3$-}braid. However, the identity $3$-braid has distinct annular Khovanov homology from $\AKh(c(\beta);\C)$, so that $n\neq 3$. It follows that $b(\beta)$ is a knot. Now, if $ \rank(\Kh(b(\beta);\C))\leq 3$ we can proceed as in the $n=3$ case and deduce that $b(\beta)$ represents a trefoil or the unknot. This is a contradiction, since the {a}nnular Khovanov homology of the corresponding braid-closures are not of the correct form. It follows that $\rank(\Kh(b(\beta);\Q))=5$. {In turn it follows from the universal coefficient theorem and}~\cite[Corollary 3.2.C]{shumakovitch2014torsion}{ that $\rank(\Kh(b(\beta);\Z/2))\geq 10$. Consider} the spectral sequence from $\AKh(b(\beta);\Z/2))$ to $\Kh(b(\beta);\Z/2))${. The proof of Lemma}~\ref{lem:akhboundmod2} {and the symmetry properties of the spectral sequences under mirroring, implies that the spectral sequence kills the generators in $(i,j,k)$ gradings given by $(0,9,3),(1,9,1),(0,7,1),(1,7,-1)$. Since the $E_\infty$-page (i.e. $\Kh(b(\beta);\Z/2))$) must have rank ten, the remaining ten generators must survive. By examining their homological and quantum gradings we}  find that $\Kh(b(\beta);\Z/2))\cong\Kh(T(2,5);\Z/2))$. It follows that $b(\beta)$ is $T(2,5)$ by~\cite[Theorem 1.1]{baldwin2021khovanov}. Birman-Menasco's classification result,~\cite{BirmanMenasco3braids}, implies that up to conjugation $\beta=\sigma^{\pm1}\sigma_2^5$. $b(\sigma^{-1}\sigma_2^5)$ has the wrong annular Khovanov homology, so the result follows.

$\mathbf{n\in\{0,2,4\}}$. In this case $b(\beta)$ has an even number of components since $\AKh(b(\beta);\C)$ is supported in even quantum gradings. Since $b(\beta)$ is a $3$-braid it has exactly two components. Now, $\rank(\Kh(b(\beta);\Z/2))\leq 8$, so by~\cite[Corollary1.4]{xie2020links} $b(\beta)$ is either
 a two{-}component unlink, $T(2,\pm 2)$ or $T(2,\pm 4)$. By Birman-Menasco's classification result~\cite{BirmanMenasco3braids}, $c(\beta)$ must be of the form $\sigma_1^{\pm1}\sigma_2^n$ for some even $|n|\leq 4$. Annular Khovanov homology distinguishes these links, concluding the proof.
\end{proof}
\subsection{Clasp-closures}\label{sec:akhclasps}

We now study the annular Khovanov homology of clasp-closures of $3$-braids. The results are dependent on the rank bound from Section~\ref{subsec:rankbound}. Our main result is the following:

\begin{theorem}\label{thm:AKHmazur}
    Annular Khovanov homology with integer coefficients detects the Mazur pattern.
\end{theorem}

To prove this we will use the following Lemma.
\begin{lemma}\label{lem:AKHdetectsclasps}
    Suppose $K$ is a clasp-closure of a $3$-braid. If $K'$ is an annular knot with ${\AKh(K;\C)\cong\AKh(K';\C)}$ then $K'$ is also a clasp-closure of a $3$-braid.
\end{lemma}

\begin{proof}

Suppose $K'$ is as in the statement of the Lemma. Consider Xie's spectral sequence from $\AKh(K';\C)$ to $\AHI(K';\C)$~\cite{XieInstantonsandannularKhovanovhomology}. Observe that the maximum non-trivial annular grading of $\AHI(K';\C)$ is either $3$ or $1$. By~\cite[Theorem 1.6]{xie_instanton_2019}, {if the maximum non-trivial grading} is one then there is a meridional surface of Euler characteristic zero, i.e. $K'$ is a wrapping number one annular link. Such links have annular Khovanov homology with maximal non-trivial annular grading one {--- which can be seen by viewing $K'$ as a connect sum of the $1$-braid and a wrapping number zero link and applying the K\"unneth formula for annular Khovanov homology ---} a contradiction.

It follows that $\AHI(K';\C)$ is of rank $2$ in annular grading $3$, the maximum annular grading in which $\AHI(K';\C)$ is non-trivial. It follows from~\cite[Proposition 8.6]{binns2024floer} that $K'$ is a clasp-braid-closure of index $3$.
\end{proof}

To prove Theorem~\ref{thm:AKHmazur} it remains to show that annular Khovanov homology distinguishes the Mazur pattern from the other clasp-closures of $3$-braids representing unknots. To that end{,} we give a partial computation for the annular Khovanov homology of the three types of clasp-closures representing unknots.

First{,} we consider the mirror of the Mazur pattern, $c(\sigma_1^{-1})$. 
\begin{lemma}\label{lem:comp1}
    $\AKh(c(\sigma^{-1});\C)$ is given by:

  \begin{center}

    \begin{tabular}{|c|c|c|c|c|}
    \hline
    \backslashbox{\!$i$\!}{\!$k$\!}
     & $-3$&$-1$&$1$&$3$ \\
\hline
$0$&$\C_{-5}$&$\C_{-3}^2$&$\C_{-1}^2$&$\C_{1}$\\\hline
$-1$&$\C_{-7}$&$\C_{-5}^2$&$\C_{-3}^2$&$\C_{-1}$\\\hline
$-2$& &$\C_{-7}$&$\C_{-5}$& \\\hline
  \end{tabular} 
\end{center}
\end{lemma}

$\AKh(c(\sigma^{-1});\Z/2)$ can be obtained by replacing every homogeneous $\C$-summand and replacing it with a $\Z/2$-summand.

\begin{proof}

Consider the $0$ and $1${-}resolutions of the crossing at the top of the diagram shown in Figure~\ref{fig:claspclosure}. The $0${-}resolution yields the $b(\sigma_1^{-1}\sigma_2^{-1})$ while the $1${-}resolution is $b(\sigma_1^{-1})$. Recall that $\AKh(b(\sigma_1^{-1}\sigma_2^{-1});\C)$ is given by:

  \begin{equation}   
\centering
\begin{tabular}{|c|c|c|c|c|}
    \hline
    \backslashbox{\!$i$\!}{\!$k$\!}
     & $-3$&$-1$&$1$&$3$ \\
\hline
$0$&$\C_{-5}$&$\C_{-3}$&$\C_{-1}$&$\C_{1}$\\\hline
$-1$& &$\C_{-5}$&$\C_{-3}$& \\\hline
  \end{tabular} \label{tab:lem315a}
  \end{equation}

This can be computed by hand. On the other hand $\AKh(b(\sigma_1^{-1});\C)$ is given by:

 \begin{equation}   
\centering
    \begin{tabular}{|c|c|c|c|c|}
    \hline
    \backslashbox{\!$i$\!}{\!$k$\!}
     & $-3$&$-1$&$1$&$3$ \\
\hline
$0$&$\C_{-4}$&$\C_{-2}^2$&$\C_{0}^2$&$\C_{2}$\\\hline
$-1$& &$\C_{-4}$&$\C_{-2}$& \\\hline
  \end{tabular} \label{tab:lem315b}
\end{equation}

Now observe that $n^-=3$, $n_0^-=2$, and we have that the exact triangle~(\ref{eq:akhskeinneg}) reduces to:

\begin{center}
\begin{tikzcd}
    \AKh(c(\sigma^{-1});\C)\arrow[rr]&&\AKh(b(\sigma_1^{-1}\sigma_2^{-1});\C)[-1]\{-2\}\arrow[dl,"\delta"]\\&\AKh(b(\sigma_1^{-1});\C)\{-1\}\arrow[ul]
\end{tikzcd}
\end{center}

{Comparing the gradings listed in table}~\ref{tab:lem315a} {and Table}~\ref{tab:lem315b}{ and noting that the connecting homomorphism $\delta$ preserves the quantum grading, we see that t}he lower right hand map in this triangle vanishes, yielding the desired result for $\AKh(c(\sigma_1^{-1});\C)$. The computation for  $\AKh(c(\sigma_1^{-1});\Z/2)$ is identical.
\end{proof}

We now give a partial computation for $\AKh(b(\sigma_1^{-3}\sigma_2\sigma_1^{-2});\C)$.

\begin{lemma}\label{lem:comp2}   $b(\sigma_1^{-3}\sigma_2\sigma_1^{-2})$ has annular Jones polynomial given by:
\begin{align*}
    t^{-3}(-q+q^{-1})+t^{-1}(-q^3+2q)+t(-q^5+2q^3)+t^3(-q^7+q^5).
\end{align*}
Moreover, in annular grading $3$ the annular Khovanov homology is rank two and {is} supported in $(i,j)$ gradings $(3,7)$ and $(2,5)$.
\end{lemma}

\begin{proof}

Consider the $0$ and $1${-}resolutions of the crossing at the top of the diagram shown in Figure~\ref{fig:claspclosure} taking $\alpha=\sigma_1^{-3}\sigma_2\sigma_1^{-2}$. The $0${-}resolution yields the braid $b'(\sigma_1^{-3}\sigma_2\sigma_1^{-2}\sigma_2^{-1})$ but with the orientation of the component which is not a braid-closure of the $1$-braid endowed with the opposite orientation --- which we have indicated with the $'$. The $1${-}resolution is $b(\sigma_1^{-3}\sigma_2\sigma_1^{-2})$.

The annular Khovanov homology of the two braids can be computed using Hunt-Keese-Licata-Morrison's program~\cite{hunt2015computing}. In particular we find that $\AKh(b(\sigma_1^{-3}\sigma_2\sigma_1^{-2}\sigma_2^{-1});\C)$ is given by:

  \begin{center}

    \begin{tabular}{|c|c|c|c|c|}
    \hline
    \backslashbox{\!$i$\!}{\!$k$\!}
     & $-3$&$-1$&$1$&$3$ \\
\hline
$0$&$\C_{-8}$&$\C_{-6}$&$\C_{-4}$&$\C_{-2}$\\\hline
$-1$& &$\C_{-8}$&$\C_{-6}$& \\\hline
$-2$& &$\C^2_{-8}$&$\C^2_{-6}$& \\\hline
$-3$& &$\C_{-10}$&$\C_{-8}$& \\\hline

$-4$& &$\C_{-12}$&$\C_{-10}$& \\\hline
$-5$& &$\C_{-14}$&$\C_{-12}$& \\\hline

  \end{tabular} 
\end{center}

To correct for the fact that one of the components is given the non-braid orientation we have to shift the homological grading by $[2]$ and the quantum grading by $\{6\}${. This holds because (annular) Khovanov homology is defined by applying a cube of resolution procedure to an unoriented (annular) link diagram, and then shifting the (diagram dependent) homological and quantum gradings by quantities determined by the orientation ---  namely the quantum grading by $\{n_+-2n_-\}$ and the homological grading by $-n_-$ where $n_\pm$ are the number of positive and negative intersections --- to obtain the diagram independent homological and quantum gradings. It is then straightforward to check that for the diagram at hand, reversing the orientation of the relevant component induces the shifts in homological and quantum grading as claimed above}.

On the other hand, $\AKh(b(\sigma_1^{-3}\sigma_2\sigma_1^{-2});\C)$ is given by:

 \begin{center}

    \begin{tabular}{|c|c|c|c|c|}
    \hline
    \backslashbox{\!$i$\!}{\!$k$\!}
     & $-3$&$-1$&$1$&$3$ \\
\hline
$1$& &$\C_{-5}$&$\C_{-3}$& \\\hline
$0$&$\C_{-7}$&$\C_{-5}^2$&$\C_{-3}^2$&$\C_{-1}$\\\hline
$-1$& &$\C_{-7}$&$\C_{-5}$& \\\hline
$-2$& &$\C_{-9}$&$\C_{-7}$& \\\hline
$-3$& &$\C_{-11}$&$\C_{-9}$& \\\hline

$-4$& &$\C_{-13}$&$\C_{-11}$& \\\hline
$-5$& &$\C_{-15}$&$\C_{-13}$& \\\hline
  \end{tabular} 
\end{center}

Now observe that $n_-=4$ and $n_-^1=6$. Thus we have the following exact triangle

\begin{center}

\begin{equation*}
\begin{tikzcd}
\AKh(L;\C)\arrow[rr]&&\AKh(b'(\sigma_1^{-3}\sigma_2\sigma_1^{-2}\sigma_2^{-1});\C)\{1\}\arrow[dl]\\&\AKh(b(\sigma_1^{-3}\sigma_2\sigma_1^{-2});\C)[3]\{8\}\arrow[ul]
\end{tikzcd}
\end{equation*}
\end{center}

 This isn't enough information to show that the exact triangle splits. However, it does split in annular gradings $\pm3$, and the decat{e}gorification of the exact triangle determines the annular Jones polynomial, as desired.\end{proof}

Let $S$ be the annular link given by the split sum of $b(\mathbf{1}_1)$ and an unknot. Observe that the annular Jones polynomial of $S$ is given by $J(S)=t^{-1}(1+q^{-2})+t(1+q^2)$. 

\begin{lemma}\label{lem:comp3}
    The annular Jones polynomial of $c(\sigma^n_1\sigma_2^{-1}\sigma_1\sigma_2)$ is given by:
      \begin{align*}
    J(c(\sigma^n_1\sigma_2^{-1}\sigma_1\sigma_2))=\dfrac{q+(-1)^{n+1}q^{1-2n}}{1+q^{2}}J(S)+(-1)^nq^{-2n}J(c(\sigma_2^{-1}\sigma_1\sigma_2)).
\end{align*}

Moreover, in annular grading $3$ the annular Khovanov homology is rank two and, supported in $(i,j)$ gradings $(-n,3-2n)$ and $(-1-n,1-2n)$.
\end{lemma}

\begin{proof}

We first compute the annular Khovanov homology of $c(\sigma_1)$, which is isotopic to $c(\sigma_2^{-1}\sigma_1\sigma_2)$.

Consider the $0$ and $1${-}resolutions of the crossing at the top of the diagram shown in Figure~\ref{fig:claspclosure} taking $\alpha=\sigma_1$. The $0${-}resolution is $b(\sigma_1\sigma_2^{-1})$. The $1${-}resolution is $b(\sigma_1)$. Observe that $n_-^0=1, n_-=2$, so that we have:

\begin{center}
\begin{tikzcd}
\AKh(c(\sigma_1);\C)\arrow[rr]&&\AKh(b(\sigma_1\sigma_2^{-1});\C)[-1]\{-2\}\arrow[dl,"\delta"]\\&\AKh(b(\sigma_1);\C)\{-1\}\arrow[ul]
\end{tikzcd}
\end{center}

Now $\AKh(b(\sigma_1\sigma_2^{-1});\C)[-1]\{-2\}$ is given by:

\begin{center}

    \begin{tabular}{|c|c|c|c|c|}
    \hline
    \backslashbox{\!$i$\!}{\!$k$\!}
     & $-3$& $-1$&$1$&$3$ \\
\hline
$0$& & $\C_{-1}$&$\C_{1}$& \\\hline
$-1$&$\C_{-5}$& $\C_{-3}^2$&$\C_{-1}^2$&$\C_1$ \\\hline
$-2$& & $\C_{-5}$&$\C_{-3}$& 
\\\hline
  \end{tabular} 
\end{center}

while $\AKh(b(\sigma_1);\C)\{-1\}$ is given by:

 \begin{center}

    \begin{tabular}{|c|c|c|c|c|}
    \hline
    \backslashbox{\!$i$\!}{\!$k$\!}
     & $-3$&$-1$&$1$&$3$ \\
\hline

$1$&&$\C_{1}$&$\C_3$&\\\hline
$0$&$\C_{-3}$&$\C_{-1}^2$&$\C_{1}^2$&$\C_3$\\\hline

  \end{tabular} 
\end{center}

Thus{, since the map $\delta$ preserves the quantum grading and increases the homological grading by one,} the exact triangle splits and $\AKh(c(\sigma_1);\C)$ is given by

\begin{center}

    \begin{tabular}{|c|c|c|c|c|}
    \hline
    \backslashbox{\!$i$\!}{\!$k$\!}
     & $-3$& $-1$&$1$&$3$ \\
\hline
$1$&&$\C_{1}$&$\C_3$&\\\hline
$0$&$\C_{-3}$&$\C_{-1}^3$&$\C_{1}^3$&$\C_3$\\\hline
$-1$&$\C_{-5}$& $\C_{-3}^2$&$\C_{-1}^2$&$\C_1$ \\\hline
$-2$&& $\C_{-5}$&$\C_{-3}$&
\\\hline

  \end{tabular} 
\end{center}

We now proceed to the general case. Remove the axis from the diagram shown in Figure~\ref{fig:inffamilyres} to obtain a diagram for the link. Observe that the {$1$-}resolution {of the highlighted crossing} is the link $S$, which has annular Khovanov homology given by:

\begin{center}

    \begin{tabular}{|c|c|c|}
    \hline
    \backslashbox{\!$i$\!}{\!$k$\!}
     & $-1$&$1$ \\
\hline

$0$&$\C_{-2}\oplus\C_{0}$&$\C_{0}\oplus\C_{2}$\\\hline

  \end{tabular} 
\end{center}

Now, since we can take $n_-^0-n_-=-1$, the exact triangle~(\ref{eq:akhskeinneg}) reduces to:

\begin{center}
\begin{tikzcd}
    \AKh(\sigma^n_1\sigma_2^{-1}\sigma_1\sigma_2;\C)\arrow[rr]&&\AKh(\sigma^{n-1}_1\sigma_2^{-1}\sigma_1\sigma_2;\C)[-1]\{-2\}\arrow[dl]\\&\AKh(S;\C)\{-1\}\arrow[ul]

\end{tikzcd}
\end{center}

Since $\AKh(S;\C)$ is trivial in annular grading $3$ this proves the second part of the result. For the first part, observe that decat{e}gorifying either of the above exact triangles we obtain;

\begin{align*}
    J(c(\sigma^n_1\sigma_2^{-1}\sigma_1\sigma_2))=q^{-1}J(S)-q^{-2}J(c(\sigma^{n-1}_1\sigma_2^{-1}\sigma_1\sigma_2)).
\end{align*}

The desired result follows by induction. \end{proof}

\begin{remark}

One could perhaps give a complete computation of the annular Khovanov homology of the infinite family of clasp-closures using techniques of J.Wang~\cite{Wangcosmetic}. The annular Jones polynomial was enough for our purposes, however, so we do not pursue this.
\end{remark}

Let $r(\beta)$ denote the reverse of the braid word $\beta$ written in terms of the standard Artin generators.

\begin{lemma}\label{lem:notwosame}
     Suppose $\beta_1$ and $\beta_2$ are $3$-braids such that $c(\beta_1)$ {and} $c(\beta_2)$ represent unknots. If ${\AKh(c(\beta_1);\C)\cong\AKh(c(\beta_2);\C)}$ then $c(\beta_1)=c(\beta_2)$ or $c(\beta_1)=c(r(\beta_2))$.
\end{lemma}
   \begin{proof}
      Observe that Lemma~\ref{lem:comp1}, Lemma~\ref{lem:comp2}, and Lemma~\ref{lem:comp3} determine the annular Khovanov homology of all of the clasp-closures up to mirroring. The annular Khovanov homology of their mirrors can be determined using formal properties of annular Khovanov homology. We can then see that no two clasp-closures of $3$-braids representing unknots have the same annular Khovanov homology in annular grading $3$ so the result follows.
   \end{proof}

\begin{proof}[Proof of Theorem~\ref{thm:AKHmazur}]
    Suppose $K$ is an annular link with $\AKh(K;\Z)\cong\AKh(c(\sigma^{-1});\Z)$. Since $\AKh(K;\Z)$ is supported in odd quantum gradings it follows that $K$ has an odd number of components. Consider the Batson-Seed link splitting sequence for $\Kh(K;\C)$. Observe that $\rank(\Kh(K;\C))\geq 2^m$, where $m$ is the number of components of $K$. Now observe that Lemma~\ref{lem:nontivialrep} {implies that the rank of $\partial_{-2}^*$ is at least $4$ since the two $V_{n-2}$ summands in the statement of the lemma are mapped non-trivially under $\partial_{-2}^*$, which we recall is part of the $\slrep(\C)$ action on annular Khovanov homology. I}n turn {it follows }that {$\rank(\Kh(K;\C)) \leq 6$.} Thus $K$ has a single component. Lemma~\ref{lem:AKHdetectsclasps} implies that $K$ is a clasp-closure of a $3$-braid.

We now show that $K$ represents the unknot. An application of the universal coefficient theorem shows that $\AKh(K;\Z/2)\cong\AKh(c(\sigma^{-1});\Z/2)$. Consider the spectral sequence from $\AKh(K;\Z/2)$ to $\Kh(K;\Z/2)$. Lemma~\ref{lem:akhboundmod2} implies that $\rank(\Kh(K;\Z/2))\leq 6$. It follows that $\rank(\Khr(K;\Q))\leq \rank(\Khr(K;\Z/2))\leq 3$, so that $L$ is either a trefoil or an unknot by~\cite{kronheimer2011khovanov} and~\cite{baldwin2018khovanov}. However, $K$ cannot be a trefoil because $\AKh(K;\C)$, and hence $\Kh(K;\C)$, does not contain a summand in quantum grading $\pm9$.

It follows that $K$ is a clasp-closure of one of Baldwin-Sivek's $3$-braid types. By Lemma~\ref{lem:notwosame}, if two such annular links have the same annular Khovanov homology then they differ only up to reversal. But of course, $c(\sigma_1^{-1})$ and $c(r(\sigma_1^{-1}))$ are isotopic, so the result follows.
\end{proof}

\bibliographystyle{alpha}
\bibliography{bibliography}

\begin{thebibliography}{DDRW08}

\bibitem[APS04]{asaeda_categorification_2004}
Marta~M. Asaeda, Jozef~H. Przytycki, and Adam~S. Sikora.
\newblock Categorification of the {Kauffman} bracket skein module of $\i$-bundles over surfaces.
\newblock {\em Algebraic \& Geometric Topology}, 4:1177--1210, 2004.

\bibitem[BB05]{birman2005braids}
Joan~S. Birman and Tara~E. Brendle.
\newblock Braids: a survey.
\newblock In {\em Handbook of knot theory}, pages 19--103. Elsevier B. V., Amsterdam, 2005.

\bibitem[BD23]{binns2022rank}
Fraser Binns and Subhankar Dey.
\newblock Rank bounds in link floer homology and detection results.
\newblock {\em Quantum Topology}, 2023.

\bibitem[BD24a]{binns2022cable}
Fraser Binns and Subhankar Dey.
\newblock Cable links, annuli and sutured {F}loer homology.
\newblock {\em Math. Res. Lett.}, 31(5):1315--1337, 2024.

\bibitem[BD24b]{binns2024floer}
Fraser Binns and Subhankar Dey.
\newblock Floer homology, clasp-braids and detection results.
\newblock {\em arXiv preprint arXiv:2405.11224}, 2024.

\bibitem[BG15]{BaldwinGrigsby}
John~A. Baldwin and J.~Elisenda Grigsby.
\newblock Categorified invariants and the braid group.
\newblock {\em Proc. Amer. Math. Soc.}, 143(7):2801--2814, 2015.

\bibitem[BHS25]{baldwin2021khovanov}
John~A. Baldwin, Ying Hu, and Steven Sivek.
\newblock Khovanov homology and the cinquefoil.
\newblock {\em J. Eur. Math. Soc. (JEMS)}, 27(6):2443--2465, 2025.

\bibitem[BLS17]{baldwin2017khovanovpointed}
John~A Baldwin, Adam~Simon Levine, and Sucharit Sarkar.
\newblock Khovanov homology and knot {F}loer homology for pointed links.
\newblock {\em Journal of Knot Theory and Its Ramifications}, 26(02):1740004, 2017.

\bibitem[BM93]{BirmanMenasco3braids}
Joan~S. Birman and William~W. Menasco.
\newblock Studying links via closed braids. {III}. {C}lassifying links which are closed {$3$}-braids.
\newblock {\em Pacific J. Math.}, 161(1):25--113, 1993.

\bibitem[BM24]{binns2020knot}
Fraser Binns and Gage Martin.
\newblock Knot {F}loer homology, link {F}loer homology and link detection.
\newblock {\em Algebr. Geom. Topol.}, 24(1):159--181, 2024.

\bibitem[BS15]{batson2015link}
Joshua Batson and Cotton Seed.
\newblock A link-splitting spectral sequence in {K}hovanov homology.
\newblock {\em Duke Mathematical Journal}, 164(5):801--841, 2015.

\bibitem[BS22]{baldwin2018khovanov}
John~A. Baldwin and Steven Sivek.
\newblock Khovanov homology detects the trefoils.
\newblock {\em Duke Math. J.}, 171(4):885--956, 2022.

\bibitem[BS25]{baldwin2022floer}
John~A. Baldwin and Steven Sivek.
\newblock Floer homology and non-fibered knot detection.
\newblock {\em Forum Math. Pi}, 13:Paper No. e1, 65, 2025.

\bibitem[BSX19]{baldwin2018khovanovhopf}
John~A. Baldwin, Steven Sivek, and Yi~Xie.
\newblock Khovanov homology detects the {H}opf links.
\newblock {\em Math. Res. Lett.}, 26(5):1281--1290, 2019.

\bibitem[BVV18]{baldwin_note_2018}
John Baldwin and David Vela-Vick.
\newblock A note on the knot {Floer} homology of fibered knots.
\newblock {\em Algebraic \& Geometric Topology}, 18(6):3669--3690, October 2018.
\newblock Publisher: Mathematical Sciences Publishers.

\bibitem[DDRW08]{dehornoy2008ordering}
Patrick Dehornoy, Ivan Dynnikov, Dale Rolfsen, and Bert Wiest.
\newblock {\em Ordering braids}.
\newblock Number 148. American Mathematical Soc., 2008.

\bibitem[Dow24]{dowlin2018spectral}
Nathan Dowlin.
\newblock A spectral sequence from {K}hovanov homology to knot {F}loer homology.
\newblock {\em Journal of the American Mathematical Society}, 2024.

\bibitem[FRW24]{farber2022fixed}
Ethan Farber, Braeden Reinoso, and Luya Wang.
\newblock Fixed-point-free pseudo-{A}nosov homeomorphisms, knot {F}loer homology and the cinquefoil.
\newblock {\em Geom. Topol.}, 28(9):4337--4381, 2024.

\bibitem[GLW18]{grigsby_annular_2018}
J.~Elisenda Grigsby, Anthony~M. Licata, and Stephan~M. Wehrli.
\newblock Annular {Khovanov} homology and knotted {Schur}-{Weyl} representations.
\newblock {\em Compositio Mathematica}, 154(3):459--502, March 2018.
\newblock Publisher: London Mathematical Society.

\bibitem[GN14]{grigsby_sutured_2014}
J.~Elisenda Grigsby and Yi~Ni.
\newblock Sutured {Khovanov} homology distinguishes braids from other tangles.
\newblock {\em Mathematical Research Letters}, 21(6):1263--1275, 2014.

\bibitem[HKLM15]{hunt2015computing}
Hilary Hunt, Hannah Keese, Anthony Licata, and Scott Morrison.
\newblock Computing annular {K}hovanov homology.
\newblock {\em arXiv preprint arXiv:1505.04484}, 2015.

\bibitem[Hos85]{hoste1985firstcoefficientoftheconwaypolynomial}
Jim Hoste.
\newblock The first coefficient of the {C}onway polynomial.
\newblock {\em Proceedings of the American Mathematical Society}, 95(2):299--302, 1985.

\bibitem[Juh06]{juhasz2006holomorphic}
Andr{\'a}s Juh{\'a}sz.
\newblock Holomorphic discs and sutured manifolds.
\newblock {\em Algebraic \& Geometric Topology}, 6(3):1429--1457, 2006.

\bibitem[Juh08]{juhasz2008floer}
Andr{\'a}s Juh{\'a}sz.
\newblock Floer homology and surface decompositions.
\newblock {\em Geometry \& Topology}, 12(1):299--350, 2008.

\bibitem[Juh10]{juhasz2010sutured}
Andr{\'a}s Juh{\'a}sz.
\newblock The sutured {F}loer homology polytope.
\newblock {\em Geometry \& Topology}, 14(3):1303--1354, 2010.

\bibitem[Kho00]{khovanov2000categorification}
Mikhail Khovanov.
\newblock A categorification of the {J}ones polynomial.
\newblock {\em Duke Math. J.}, 101(3):359--426, 2000.

\bibitem[KLO]{KLO}
{\em KLO ({K}not-{L}ike {O}bjects) software, version 0.979 alpha}.
\newblock {https://community.middlebury.edu/~mathanimations/klo/}.

\bibitem[KM11]{kronheimer2011khovanov}
Peter~B Kronheimer and Tomasz~S Mrowka.
\newblock Khovanov homology is an unknot-detector.
\newblock {\em Publications math{\'e}matiques de l'IH{\'E}S}, 113(1):97--208, 2011.

\bibitem[kno]{knotatlas}
The {K}not {A}tlas.
\newblock \url{https://katlas.org/wiki/Main_Page}.

\bibitem[Lee05]{lee2005endomorphism}
Eun~Soo Lee.
\newblock An endomorphism of the {K}hovanov invariant.
\newblock {\em Advances in Mathematics}, 197(2):554--586, 2005.

\bibitem[LM24]{linkinfo}
Charles Livingston and Allison~H. Moore.
\newblock Linkinfo: Table of link invariants.
\newblock URL: \url{linkinfo.math.indiana.edu}, Current Month 2024.

\bibitem[Mar22]{martin2022khovanov}
Gage Martin.
\newblock Khovanov homology detects $ t (2, 6) $.
\newblock {\em Mathematical Research Letters}, 29(3):835--850, 2022.

\bibitem[Mur74]{Murasugiclosed3}
Kunio Murasugi.
\newblock {\em On closed {$3$}-braids}, volume No. 151 of {\em Memoirs of the American Mathematical Society}.
\newblock American Mathematical Society, Providence, RI, 1974.

\bibitem[Ni20]{ni2020exceptional}
Yi~Ni.
\newblock Exceptional surgeries on hyperbolic fibered knots.
\newblock {\em arXiv preprint arXiv:2007.11774}, 2020.

\bibitem[OS04]{Holomorphicdisksandknotinvariants}
Peter Ozsv\'{a}th and Zolt\'{a}n Szab\'{o}.
\newblock Holomorphic disks and knot invariants.
\newblock {\em Advances in Mathematics}, 186(1):58--116, August 2004.

\bibitem[OS08a]{ozsvath2008holomorphic}
Peter Ozsv{\'a}th and Zolt{\'a}n Szab{\'o}.
\newblock Holomorphic disks, link invariants and the multi-variable {A}lexander polynomial.
\newblock {\em Algebraic \& Geometric Topology}, 8(2):615--692, 2008.

\bibitem[OS08b]{HolomorphicdiskslinkinvariantsandthemultivariableAlexanderpolynomial}
Peter Ozsv\'{a}th and Zolt\'{a}n Szab\'{o}.
\newblock Holomorphic disks, link invariants and the multi-variable {Alexander} polynomial.
\newblock {\em Algebraic \& Geometric Topology}, 8(2):615--692, May 2008.
\newblock Publisher: Mathematical Sciences Publishers.

\bibitem[OS08c]{ozsvath2008linkFloerThurstonnorm}
Peter Ozsv{\'a}th and Zolt{\'a}n Szab{\'o}.
\newblock Link {F}loer homology and the {T}hurston norm.
\newblock {\em Journal of the American Mathematical Society}, 21(3):671--709, 2008.

\bibitem[Pla06]{Plamenevskaya06transverse}
Olga Plamenevskaya.
\newblock Transverse knots and {K}hovanov homology.
\newblock {\em Math. Res. Lett.}, 13(4):571--586, 2006.

\bibitem[Ras03]{Rasmussen}
Jacob Rasmussen.
\newblock Floer homology and knot complements.
\newblock {\em arXiv:math/0306378}, June 2003.
\newblock arXiv: math/0306378.

\bibitem[Shu14]{shumakovitch2014torsion}
Alexander Shumakovitch.
\newblock Torsion of {K}hovanov homology.
\newblock {\em Fundamenta Mathematicae}, 225(1):343--364, 2014.

\bibitem[SZ24]{stoffregen2018localization}
Matthew Stoffregen and Melissa Zhang.
\newblock Localization in {K}hovanov homology.
\newblock {\em Geom. Topol.}, 28(4):1501--1585, 2024.

\bibitem[Wan22]{Wangcosmetic}
Joshua Wang.
\newblock The cosmetic crossing conjecture for split links.
\newblock {\em Geom. Topol.}, 26(7):2941--3053, 2022.

\bibitem[Xie21]{XieInstantonsandannularKhovanovhomology}
Yi~Xie.
\newblock Instantons and annular {K}hovanov homology.
\newblock {\em Adv. Math.}, 388:Paper No. 107864, 51, 2021.

\bibitem[XZ19]{xie_instanton_2019}
Yi~Xie and Boyu Zhang.
\newblock Instanton {Floer} homology for sutured manifolds with tangles.
\newblock {\em arXiv:1907.00547 [math]}, July 2019.
\newblock arXiv: 1907.00547.

\bibitem[XZ22]{xie2020links}
Yi~Xie and Boyu Zhang.
\newblock On links with {K}hovanov homology of small ranks.
\newblock {\em Math. Res. Lett.}, 29(4):1261--1277, 2022.

\end{thebibliography}

\end{document}